\documentclass[11pt]{article}
\usepackage[utf8]{inputenc}
\usepackage[T1]{fontenc}
\usepackage{amssymb,mathtools,mathrsfs,amsthm}
\usepackage{graphicx}
\usepackage{babel}
\usepackage[]{hyperref}
\usepackage{microtype}
\usepackage{caption}
\usepackage{subcaption}
\usepackage{amsfonts,ulem,amssymb,amsmath, xcolor}
\usepackage{bbold}
\usepackage{enumitem}
\usepackage{tikz}
\usepackage{tabularx}
\usepackage{appendix}
\usepackage{xcite}
\usepackage[margin=1in]{geometry}
\newtheorem{theorem}{Theorem}
\newtheorem{assumption}{Assumption}

\newtheorem{lemma}{Lemma}
\newtheorem{proposition}{Proposition}
\newtheorem{definition}{Definition}
\newtheorem{corollary}{Corollary}
\newtheorem{remark}{Remark}

\newcommand{\bx}{\boldsymbol{x}}
\newcommand{\by}{\boldsymbol{y}}

\newcommand{\re}{\varepsilon}

\renewcommand{\vec}[1]{\boldsymbol{#1}}

\title{What does it mean for a 3D star-shaped scatterer to be small in the time domain?}
\author{M. Kachanovska, A. Savchuk}
\date{}
\begin{document}

	\maketitle\begin{abstract}
In the frequency domain wave scattering problems, obstacles can be effectively replaced by point scatterers as soon as the wavelength of the incident wave exceeds significantly their diameter. The situation is less clear in the time domain, where recent works suggest the presence of an additional temporal scale that quantifies the smallness of the obstacle. In this paper we argue that this is not necessarily the case, and that it is possible to construct asymptotic models with an error that does not deteriorate in time, at least in the case of a sound-soft scattering problem by a star-shaped obstacle in 3D.  
	\end{abstract}
	
\section{Introduction and problem setting}
	\subsection{Introduction}
	This work is motivated by a question and analysis posed in the recent paper by P. Martin \cite{martin}: when considering wave scattering by a particle in the time domain, what does it mean for a particle to be 'small'? Limiting the discussion to the three-dimensional case,  it is well-known that in the frequency domain, for the problem of wave scattering of a wave of wavelength $\lambda$ by a particle of characteristic size $\varepsilon$, one can effectively replace the particle by a point scatterer as soon as $\varepsilon\ll \lambda$, cf. \cite[Section 8.1]{martin_monograph}. Under these conditions, one can expect the asymptotic model to produce an absolute error in the scattered far-field of $O(\varepsilon(\varepsilon/\lambda)^p)$, for some $p\geq 0$, with hidden constants independent of $\lambda$ and $\varepsilon$. Since the magnitude of the scattered field is $O(\re)$, a particle is 'small' if $\varepsilon\ll \lambda$. 
	
	 In the time-domain, there are two parameters: the smallest wavelength $\lambda_{min}$ of the approximately band-limited signal and the final simulation time $T$. In other words, one can expect that the absolute $L^{\infty}(0,T)$-error of an asymptotic model behaves as $O(\varepsilon (\varepsilon/\lambda_{min})^p T^q)$, where $q>0$, see  e.g. \cite{mk_as}.  Although in many works, cf. \cite{sini_wang_yao} for the Foldy-Lax model for scattering by a cluster of small particles, \cite{barucq} for scattering by a small sphere, \cite{phd_matessi} by a small particle in $\mathbb{R}^3$, or  \cite{korikov} by a small particle inside a bounded domain, the estimates are stated without explicit dependence in the final time, from the proofs of these estimates one can deduce the degeneration of the error bounds with time as $t\rightarrow +\infty$.
		 
	Considering a more general class of asymptotic time-dependent problems, the estimates of the above type were also obtained in the context of Minnaert-type resonances in \cite[Theorem 1.1]{li_sini}, \cite[Theorem 5.3] {mantile_posilicano}. An alternative approach is taken in the paper \cite{baldassari_millien_vanel} modelling plasmonic resonances. There, the approximations are based on modal expansions; in this case an extra scale is introduced by an abrupt truncation of the Fourier-transformed signal in frequency. The error estimates are then  based on an interplay between the simulation time, size of the particle and the cut-off frequency. 
	
	Importantly, all these estimates seem to indicate that one has to assign an additional temporal scale to define the 'smallness' of the particle. In other words, we expect the particle to be small if $\re\ll \lambda T^{-\alpha}$, for some $\alpha>0$, where $T$ is a final simulation time. 
	 
	 In this work we concentrate on sound-soft scattering by a small particle, and prove that such a deterioration of the error in time is not intrinsic to asymptotic models that replace a particle by a point scatterer, and that it is possible to construct a model with an error not deteriorating in time, at least in the situation of a single scatterer. The mathematical meaning of 'not deteriorating in time' will be explained in the course of the article. The reason for this is that, roughly speaking, in 3D, the local energy of the scattered field decays fast, and constructing an accurate asymptotic model mimicking such a behaviour is not a difficult task. We provide quantitative bounds on the local energy of the field scattered by a small particle and the error of the field when approximating the particle by a point source valid for long times.
\subsection{Problem setting}
	\label{sec:pb_setting}
	Let $\Omega$ be a smooth domain in $\mathbb{R}^3$ and  $\Omega^c:=\mathbb{R}^3\setminus\overline{\Omega}$ its exterior. We assume that $\Omega\subseteq B_1(\vec{0})$, which implies that $\operatorname{diam}\Omega\leq 1$. Additionally, let $\Omega^c$ be connected. Without loss of generality, let $\vec{0}\in \Omega$. The family of the domains $\Omega^{\re}$ is defined by applying a contraction to $\Omega$: $$\Omega^{\varepsilon}:=\{\vec{x}^{\varepsilon}=\varepsilon\vec{x}, \; \vec{x}\in \Omega\}, \quad \Gamma^{\re}:=\partial\Omega^{\re}, \, 0<\varepsilon\leq 1.$$
	We study the wave equation, posed in the exterior $\Omega^{\varepsilon,c}=\mathbb{R}^3\setminus\overline{\Omega^{\re}}$, more precisely, we look for $u^{\re},\, v^{\re}: \mathbb{R}_{\geq 0}\times \Omega^{\re,c}\rightarrow \mathbb{R}$, s.t. 
	\begin{align}
		\label{eq:wave}
		\begin{split}
		&\partial_t u^{\re}- v^{\re}=0,\qquad 
		\partial_t v^{\re}-\Delta u^{\re}=0\text{ in }\mathbb{R}_{>0}\times \Omega^{\re,c},\\
		&\gamma_0^{\re}u^{\re}\equiv \left. u^{\re}\right|_{\Gamma^{\re}}=0,\quad
		u^{\re}(0,\vec{x})=u_0(\vec{x}), \quad  v^{\re}(0,\vec{x})=v_0(\vec{x}).
		\end{split}
	\end{align}
The initial data $u_0, \, v_0$ is sufficiently regular, defined in all of $\mathbb{R}^3$, but supported away from the scatterer.  
More precisely, let $R_0>r_0>1$, and $\mathcal{O}_0:=B_{R_0}(0)\setminus \overline{B_{r_0}(0)}$ be a spherical shell enclosing $\overline{\Omega}$. We assume that, with a regularity exponent $k_{reg}\in \mathbb{N}_0$,
	\begin{align}
		\label{eq:ic}
	v_0\in H^{k_{reg}}(\mathbb{R}^3), \quad u_0\in H^{k_{reg}+1}(\mathbb{R}^3), \quad \operatorname{supp}u_0,\; \operatorname{supp}v_0\subset \mathcal{O}_0.
	\end{align}
The above implies in particular that  $\left. u_0\right|_{\Omega^{\re,c}}\in H^{k_{reg}+1}_0(\Omega^{\re,c})$. 

In what follows, we will employ the notation $\vec{u}^{\re}:=\left(\begin{matrix}
	u^{\re}\\
	v^{\re}
\end{matrix}\right)$, $\vec{u}_{0}:=\left(\begin{matrix}
u_{0}\\
v_{0}
\end{matrix}\right)$. 
We have the following energy identity, obtained by testing the first equation in \eqref{eq:wave} by  $\Delta u^{\re}$, the second equation by $v^{\re}$, and next integrating by parts:
\begin{align*}
\frac{d}{dt}{E}^{\re}=0, \text{ where } {E}^{\re}(t):=\frac{1}{2}\left(\|\nabla u^{\re}(t)\|^2_{L^2(\Omega^{\re,c})}+\|v^{\re}(t)\|_{L^2(\Omega^{\re,c})}^2\right).
\end{align*}
Thus, the total energy remains constant. On the other hand, the local energy
\begin{align*}
{E}^{\re}(\vec{u}^{\re}(t); K):=\frac{1}{2}\left(\|\nabla u^{\re}(t)\|^2_{L^2(K)}+\|v^{\re}(t)\|_{L^2(K)}^2\right)
\end{align*}
defined on a \textit{bounded} set $K\subset \mathbb{R}^3$ is known to behave differently. From the deep results of \cite{burq}, it follows that for sufficiently regular initial data ${E}^{\re}(\vec{u}^{\re}(t); K)$ decays as $t\rightarrow +\infty$, with a rate of decay being defined by the geometry of the obstacle and the regularity of the data. In particular, when in \eqref{eq:ic} $k_{reg}=1$, it holds that
$
	{E}^{\re}(\vec{u}^{\re}(t); K)\leq C(K;\Omega^{\re})\log^{-1}t.$ 
 This estimate can be improved for star-shaped obstacles \cite{morawetz, lax_phillips,melrose} to $	{E}^{\re}(\vec{u}^{\re}(t); K)\leq C(K;\Omega^{\re})\mathrm{e}^{-c_{\re}t}.$ 

In all these cases the dependence of the estimates on the parameter $\varepsilon$ does not seem to be explicit. On the other hand, it is expected that when $\varepsilon\rightarrow 0$, the field $\vec{u}^{\re}(\vec{x},t)$ should approach the solution $\vec{u}^{inc}$ of the initial-value problem \eqref{eq:wave}  posed in the space $\mathbb{R}^3$ (i.e. the obstacle 'disappears'), for which the strong Huygens principle holds true: the local energy  $E^{\re}(\vec{u}^{inc}(t); K)=0$ for $t>t_K>0$. Thus, the effect of the obstacle on the time-dependence of the field $\vec{u}^{\re}$ can be  measured by studying the scattered field, namely the difference between the solution of the Cauchy-problem in the free space and the solution to \eqref{eq:wave}. We formalize it below.

First, let $\vec{u}^{inc}: \, \mathbb{R}_{\geq 0}\times \mathbb{R}^3\rightarrow \mathbb{R}$ be a solution to the free space wave equation:
\begin{align}
	\label{eq:incident}
	\begin{split}
		&\partial_t u^{inc}- v^{inc}=0,\qquad \partial_t v^{inc}-\Delta u^{inc}=0\text{ in }\mathbb{R}_{>0}\times \mathbb{R}^3,\\
		&u^{inc}(0,\vec{x})=u_0, \quad  v^{inc}(0,\vec{x})=v_0.
	\end{split}
\end{align}
The scattered field $\vec{u}_{sc}^{\re}:=\vec{u}^{\re}-\left.\vec{u}^{inc}\right|_{\Omega^{\varepsilon,c}}$ then solves the boundary-value problem:
\begin{align}
	\label{eq:scattered_field}
	\begin{split}
		&\partial_t u^{\re}_{sc}- v^{\re}_{sc}=0,\qquad 
		\partial_t v^{\re}_{sc}-\Delta u^{\re}_{sc}=0\text{ in }\mathbb{R}_{>0}\times \Omega^{\re,c},\\
		&\gamma_0^{\re}u^{\re}_{sc}=-\gamma_0^{\varepsilon} u^{inc},\qquad
		u^{\re}_{sc}(0,\vec{x})=0, \quad  v^{\re}_{sc}(0,\vec{x})=0.
	\end{split}
\end{align}
\subsection{Questions studied in the paper}
In this manuscript we will restrict our attention (see Remark \ref{rem:reason}) to a sub-class of obstacles satisfying a so-called non-trapping assumption, see Proposition 13.3.1 in \cite{petkov_stoyanov}.
\begin{assumption}
	\label{assumption:star_shaped}
	$\Omega$ is a $C^{\infty}$ star-shaped obstacle; more precisely, there exists $\vec{x}_0\in \Omega$, s.t. for all $\hat{s}\in \mathbb{S}^2$ there exists a unique  $r_{\hat{s}}>0$, s.t. $\vec{x}_0+r_{\hat{s}}\hat{s}\in \partial\Omega$. Moreover, we assume additionally that the normal curvature of $\Omega$ does not vanish to an infinite order in any direction.
\end{assumption}
We fix $R_{\operatorname{ff}}>r_{\operatorname{ff}}>1$, and define the far-field region
\begin{align*}
K_{\operatorname{ff}}:=B_{R_{\operatorname{ff}}}(\vec{0})\setminus\overline{ B_{r_{\operatorname{ff}}}(\vec{0})}, \text{ so that }\operatorname{dist}({\Omega}^{\re}, K)\geq r_{\operatorname{ff}}-\re.
\end{align*}
\textbf{Question 1. How does the rate of decay of the $L^2$-norm of the solution $\|u^{\re}_{sc}(t)\|_{L^2(K_{\operatorname{ff}})}$ (i.e. measured locally) depend on $t$ and $\re$? }

 Our answer to this question is given in Theorem \ref{theorem:answer1}. Remark that in what follows we concentrate solely on one component $u^{\re}_{sc}$  of the field $\vec{u}^{\re}_{sc}$.

Recent research in the field of time-domain asymptotic methods has given rise to new models, cf. e.g. \cite{sini_wang_yao, mk_as,barucq,phd_matessi}, for the problem \eqref{eq:wave}. Such models allow to simplify the computations significantly, cf. \cite{barucq}, by incorporating  important information about the dynamics of the problem \eqref{eq:wave}, however, up to our knowledge, the question of the long-time behaviour of the errors has not so far been addressed for any of the works. We do so for the one-particle asymptotic model proposed in \cite{sini_wang_yao}. Denoting by $\vec{u}^{\re}_{app}$ the solution given by an asymptotic model, and by ${e}^{\re}:= {u}^{\re}_{app}-{u}^{\re}$ the error between the $u^{\re}$-components of the exact and approximate solution, we aim at answering the following question. 

\textbf{Question 2. How does the local error of the solution  $\|{e}^{\re}_{\operatorname{app}}\|_{L^2({K_{\operatorname{ff}}})}$ depend on $t$ and $\re$?}

An answer to this question can be found in Theorem \ref{theorem:answer2}.

This paper is organized as follows. In Section \ref{sec:lcl} we describe a strategy to estimate the local energy decay with the help of the resolvent estimates in the Laplace domain. Next we obtain a suitable estimate on the resolvent; the key ingredients are existing estimates in the $\re$-independent case, a scaling argument and a suitable Krein-type resolvent formula; those estimates are then translated into the time domain. Section  \ref{sec:asymptotic} is dedicated to derivation of such estimates for two chosen asymptotic models.
\section{Local energy decay estimates}
\label{sec:lcl}
Let us define the energy of the initial data, which quantifies the regularity of the incident field: 
\begin{align*}
	E_0^{k_{reg}}:=	\frac{1}{2}\left(\|u_0\|_{H^{k_{reg}+1}(\mathbb{R}^3)}^2+\|v_0\|_{H^{k_{reg}}(\mathbb{R}^3)}^2\right).
\end{align*}
The key result of this section reads. 
\begin{theorem}\label{theorem:answer1}
Assume that $k_{reg}\geq 6$. 
There exists $\re_0>0$, s.t. for all $0<\varepsilon<\re_0$, for the problem described in Section \ref{sec:pb_setting}, the following energy decay bounds on the scattered field hold true. With $\gamma_{\Omega}>0$ that depends on the obstacle only, the scattered field measured in the far-field region $K_{\operatorname{ff}}$ satisfies, with  $t=R_0+R_{\operatorname{ff}}+\tau$, 
\begin{align*}
	\|u^{\re}_{sc}(t)\|_{L^{2}(K_{\operatorname{ff}})}\leq 		{C}_{\operatorname{ff}}	\left\{
	\begin{array}{ll}
	 \re   \sqrt{E_0^{k_{reg}}},& \tau\leq 0,	 		\\	
	 \re^{k_{reg}-1}\mathrm{e}^{-\frac{\gamma_{\Omega}}{\re}\tau}\sqrt{E_0^{k_{reg}}}, &\tau>0,
	\end{array}
	\right.
\end{align*}
where the constant $C_{\operatorname{ff}}$ depends on $R_{\operatorname{ff}}$, $r_{\operatorname{ff}}$, $R_0$, $r_0$, $k_{reg}$, $\gamma_{\Omega}$, but is independent of $\varepsilon$, $t$. 
\end{theorem}
Let us comment on the above result. For small times, the scattered field behaves as $O(\re)$, which is consistent with the known asymptotic frequency-domain behaviour \cite[Section 8.2]{martin_monograph}. Here is what happens for longer times. The wave emitted by the source  $\vec{u}_0$ located in $\mathcal{O}_0=B_{R_0}\setminus\overline{B_{r_0}}$, reaches the boundary of the small obstacle at the time $t\approx R_{0}$ (remark that the speed of the wave propagation is $1$), next gets reflected and reaches the farthest observation point on $\partial B(\vec{0},R_{\operatorname{ff}})$ at $t_*=R_0+R_{\operatorname{ff}}+O(\re)$, where $O(\re)$ accounts for the size of the obstacle, and finally decays exponentially fast with $(t-t_*)$. The decay rate depends on the obstacle size, and is more pronounced as $\re\rightarrow 0$. In other words, the behaviour is getting close to the free space case, where the wave should vanish after this point in time due to the Huygens principle. To understand the exponential decay rate stated in the theorem, we remark that $\mathrm{e}^{-\frac{\gamma_{\Omega}}{\re}(t+\frac{1}{\gamma_{\Omega}}m\re\log \re)}=\mathrm{e}^{-\frac{\gamma_{\Omega}}{\re}(t+\frac{1}{\gamma_{\Omega}}m\re\log \re^{-1})}=\re^m\mathrm{e}^{-\frac{\gamma_{\Omega}}{\re}t}$, in other words, after the time $m\re\log\re^{-1}$ the wave becomes by a factor of $O(\re^m)$ smaller. An interesting by-product of the analysis is the effect of the regularity of the field on the subsequent behaviour of the wave, which shows that higher regularity fields seem to have a significantly smaller amplitude after $t_*$ than fields of a lower regularity. 
\begin{remark}
	\label{rem:reason}
	Up to the best of the authors' knowledge, in the literature most estimates of such type are obtained using either of the two approaches: based on the  Morawetz estimates \cite{morawetz_first,morawetz,morawetz_ralston_strauss,lax_morawetz_phillips,lax_morawetz_phillips2} (for star-shaped obstacles), or by Laplace contour deformation techniques in the spirit of \cite{burq}, that, however, rely on the use of bounds on the cut-off resolvent of the underlying operator. Obtaining such bounds in the case when scatterers are 'trapping' relies on microlocal/semi-classical analysis techniques, and optimal estimates depend on the region $K$ where the scattered wave is measured, see \cite{burq2}. The effect of trapping can be shown to be less pronounced in the far-field then in the near-field. We do not believe that quite elementary techniques used in this paper allow to account properly for such subtle effects in this case, but think that our bounds are close to optimal in the non-trapping case of Assumption \ref{assumption:star_shaped}.
\end{remark}
\subsection{Strategy of the proof of Theorem \ref{theorem:answer1} and auxiliary notation}
We follow the following approach as used e.g. in  \cite{burq}. We analyze the analyticity properties of the Laplace transform of the scattered field and exploit them in a deformation of the inverse Laplace integral, to obtain close to optimal in the time domain estimates. These analyticity properties depend in particular on the geometry of the domain $\Omega^{\varepsilon}$. 
\subsubsection{Representation formula}
Let us define a family of operators on the space $V^{\re}:=H^1_0(\Omega^{\re,c})\times L^2(\Omega^{\re,c})$:
\begin{align*}
	\mathbb{B}^{\re}:=\left(
	\begin{matrix}
		0& \operatorname{Id}\\
		\Delta & 0
	\end{matrix}
	\right),\quad D(\mathbb{B}^{\re})=\{\vec{v}=\left(\begin{matrix}
		v_1\\
		v_2 \end{matrix}\right)\in V^{\re}: \, \Delta v_1\in L^2(\Omega^{\re,c}), \, v_2\in H^1_0(\Omega^{\re, c})\}.
\end{align*}
The problem \eqref{eq:wave} satisfied by $\vec{u}^{\re}$ rewrites as a Cauchy problem
\begin{align*}
	&\frac{d}{dt}\vec{u}^{\re}-\mathbb{B}^{\re}\vec{u}^{\re}=0,\quad t>0, \quad \vec{u}^{\re}(0)=\left.\vec{u}_0\right|_{\Omega^{\re,c}}.
\end{align*}
Applying the Laplace transform, defined as follows for $\varphi\in L^1(0,\infty)$,
\begin{align*}
	\hat{\varphi}(s):=\int_0^{\infty}\mathrm{e}^{-st}\varphi(t)dt, \quad \Re s\geq 0,
\end{align*}
to the above expression yields the problem satisfied by the Laplace transform of $\hat{\vec{u}}^{\re}$:
\begin{align*}
	&(s-\mathbb{B}^{\re})\hat{\vec{u}}^{\re}=\left.\vec{u}_0\right|_{\Omega^{\re,c}}.
\end{align*}
Our starting point is the following representation formula for the solution of \eqref{eq:wave}: 
\begin{align}
	\label{eq:representation_formula}
	\vec{u}^{\re}(t)=\frac{1}{2\pi i}\int_{i\mathbb{R}+\mu}\frac{\mathrm{e}^{st}}{(\mu_*-s)^p}(s-\mathbb{B}^{\re})^{-1}(\mu_*-\mathbb{B}^{\re})^p\left.\vec{u}_0\right|_{\Omega^{\re,c}}ds, \quad 0<\mu<\mu_*
\end{align}
where $p>1$, 
as proven in \cite{burq} or \cite[Lemma 3.3]{vacossin}. Taking formally $p=0$ in the above, we recognize the Bromwich inversion formula. Choosing $p>1$,  provided that the initial conditions are sufficiently regular, i.e. $k_{reg}\geq p$ in \eqref{eq:ic}, ensures  convergence of the above integral as a Bochner integral of a function from $V^{\re}$.

We would like to rewrite this representation formula for the scattered field. We will start with the incident field. Let us introduce a family of operators on the space which we denote with an abuse of notation by  $V^{0}:=H^1_0(\mathbb{R}^3)\times L^2(\mathbb{R}^3)$:
\begin{align*}
	\mathbb{B}:=\left(
	\begin{matrix}
		0& \operatorname{Id}\\
		\Delta & 0
	\end{matrix}
	\right),\quad D(\mathbb{B})=\{\vec{v}=\left(\begin{matrix}
		v_1\\
		v_2 \end{matrix}\right)\in V^0: \, \Delta v_1\in L^2(\mathbb{R}^3), \, v_2\in H^1_0(\mathbb{R}^3)\}.
\end{align*}
Then the incident field can be written analogously to \eqref{eq:representation_formula}:
\begin{align}
	\label{eq:representation_formula_inc}
	\vec{u}^{inc}(t)=\frac{1}{2\pi i}\int_{i\mathbb{R}+\mu}\frac{\mathrm{e}^{st}}{(\mu_*-s)^p}(s-\mathbb{B})^{-1}(\mu_*-\mathbb{B})^p\vec{u}_0ds, \quad 0<\mu<\mu_*, \quad p>1.
\end{align}
We will compare the total field to the incident field; since the latter one is defined on the $\mathbb{R}^3$ and the former only on $\Omega^{\re,c}$, we introduce the operator of extension by zero:  
\begin{align}
	\mathbb{E}^{\re}: \, L^2(\Omega^{\re,c})\rightarrow L^2(\mathbb{R}^3), \quad \mathbb{E}^{\re}u=u\mathbb{1}_{\Omega^{\re,c}}.
\end{align}
Similarly, let $\mathbb{R}^{\re}$ be its left-inverse, namely, an operator of restriction to $\Omega^{\re,c}$, i.e. 
\begin{align*}
	\mathbb{R}^{\re}:\, L^2(\mathbb{R}^3)\rightarrow L^2(\Omega^{\re,c}), \quad \mathbb{R}^{\re}u=\left. u\right|_{\Omega^{\re,c}}. 
\end{align*}
With these notations, it holds that (recall that $\mathbb{B}$ and $\mathbb{B}^{\re}$ are local operators, expressed via the Laplacian): $\mathbb{R}^{\re}(z-\mathbb{B})^k\vec{u}_0=(z-\mathbb{B}^{\re})^k\mathbb{R}^{\re}\vec{u}_0$, $k\in \mathbb{N}$. Then the $\vec{u}^{\re}_{sc}$ writes
\begin{align}
\nonumber
	\vec{u}^{\re}_{sc}(t)&=\frac{1}{2\pi i}\int_{i\mathbb{R}+\mu}\frac{\mathrm{e}^{st}}{(\mu_*-s)^p}\left((s-\mathbb{B}^{\re})^{-1}\mathbb{R}^{\re}-\mathbb{R}^{\re}(s-\mathbb{B})^{-1}\right)(\mu_*-\mathbb{B})^p\vec{u}_0ds\\
	\label{eq:representation_formula_scat2}
	&= \frac{1}{2\pi i}\int_{i\mathbb{R}+\mu}\frac{\mathrm{e}^{st}}{(\mu_*-s)^p}\left((s-\mathbb{B}^{\re})^{-1}-\mathbb{R}^{\re}(s-\mathbb{B})^{-1}\mathbb{E}^{\re}\right)\mathbb{R}^{\re}(\mu_*-\mathbb{B})^p\vec{u}_0ds.
\end{align}
To proceed, we will need a more convenient expression of $(s-\mathbb{B}^{\re})^{-1}, \, (s-\mathbb{B})^{-1}$.

\subsubsection{Convenient expressions of resolvents of $\mathbb{B}^{\re}$, $\mathbb{B}$}
Let 
\begin{align*}
	\Delta_D^{\re}v=\Delta v=\sum_{j}\partial_{x_j}^2v, \quad D(\Delta_D^{\re}):=H^1_0(\Delta; \Omega^{\re,c})=H^1_0(\Omega^{\re,c})\cap H^2(\Omega^{\re,c})
\end{align*}
be the usual Dirichlet Laplacian. 
For $s\in \mathbb{C}^+:=\{z\in \mathbb{C}: \, \Re z>0\}$, the resolvent of the Dirichlet Laplacian written with respect to $s^2$ is defined as $\mathcal{R}_D^{\re}(s):=(s^2-\Delta_D^{\re})^{-1}$. It is analytic as $\mathcal{L}(L^2(\Omega^{\re,c}), H^1_0(\Omega^{\re,c}))$-valued function on $\mathbb{C}^+$, because $\Delta^{\re}_D$ is non-positive. The resolvent of $\mathbb{B}^{\re}$ can be expressed solely via $\mathcal{R}^{\re}_D$; the proof is left to the reader.  
\begin{lemma}[A convenient expression of the resolvent of $\mathbb{B}^{\re}$]
	\label{lem:resolvent_Bre}
	For all $s\in \mathbb{C}^+$,
	$$(s-\mathbb{B}^{\re})^{-1}=\left(
	\begin{matrix}
		s\mathcal{R}_D^{\re}(s)& \mathcal{R}_D^{\re}(s)\\
		s^2\mathcal{R}_D^{\re}(s)-\operatorname{Id} & s\mathcal{R}_D^{\re}(s)
	\end{matrix}\right).$$
\end{lemma}
From Lemma \ref{lem:resolvent_Bre}, we see that analyticity properties of the operator-valued function $s\mapsto (s-\mathbb{B}^{\re})^{-1}$  are inherited from the analyticity properties of $s\mapsto\mathcal{R}_D^{\re}(s)$.

To write an analogous expression for the operator $\mathbb{B}$, we will need the resolvent w.r.t. $s^2$ of the free Laplacian
\begin{align}
	\label{eq:resolvent_free}
\mathcal{R}_0: \, L^2(\mathbb{R}^3)\rightarrow L^2(\mathbb{R}^3), \quad \mathcal{R}_0(s)=(s^2-\Delta_0)^{-1}, \quad D(\Delta_0)=H^2(\mathbb{R}^3).
\end{align}
Then we have the identity analogous to Lemma \ref{lem:resolvent_B} holding true.
\begin{lemma}[A convenient expression of the resolvent of $B$]
	\label{lem:resolvent_B}
	For all $s\in \mathbb{C}^+$,
	$$(s-\mathbb{B})^{-1}=\left(
	\begin{matrix}
		s\mathcal{R}_0(s)& \mathcal{R}_0(s)\\
		s^2 \mathcal{R}_0(s)-\operatorname{Id} & s\mathcal{R}_0(s)
	\end{matrix}\right).$$
\end{lemma}
In view of \eqref{eq:representation_formula_scat2}, 
let us introduce a restriction of the free resolvent to $\Omega^{\varepsilon,c}$, itself acting on functions supported on $\Omega^{\varepsilon,c}$. In other words, let 
\begin{align*}
	\mathcal{R}_0^{\re}(s):=\mathbb{R}^{\re}\mathcal{R}_0(s)\mathbb{E}^{\re}: \, L^2(\Omega^{\varepsilon,c})\rightarrow L^2(\Omega^{\varepsilon,c}).
\end{align*}
Combining Lemmas \ref{lem:resolvent_Bre} and \ref{lem:resolvent_B}, as well as the above definition, into \eqref{eq:representation_formula_scat2}, and defining  
\begin{align*}
	\vec{u}_0^{p}:=(\mu_*-\mathbb{B})^p\vec{u}_0,\quad 
\end{align*}
and using the same notation for the restriction of this function to $\Omega^{\re,c}$, 
allows us to obtain an explicit expression for the first component of the vector $\vec{u}^{\re}_{sc}$:
\begin{align}
	\label{eq:representation_formula_main}
	u^{\varepsilon}_{sc}(t)=\frac{1}{2\pi i}\int_{i\mathbb{R}+\mu}\frac{\mathrm{e}^{st}}{(\mu_*-s)^p}(\mathcal{R}^{\varepsilon}_D(s)-\mathcal{R}_0^{\re}(s))(su_0^{p}+v_0^{p})ds. 
\end{align}
Let us remark that the identity \eqref{eq:representation_formula_main} can be obtained by considering the second-order wave equation for $u^{\varepsilon}(t), \, u^{inc}(t)$, cf. \cite{mantile_posilicano}.  \begin{remark}
	\label{rem:bmap}
By a direct computation, $$(\mu_*-\mathbb{B})^p: \, H^{k+1}(\mathbb{R}^3)\times H^{k}(\mathbb{R}^3)\rightarrow H^{k+1-p}(\mathbb{R}^3)\times H^{k-p}(\mathbb{R}^3).$$
\end{remark}
Because we are interested in the behaviour of  $u^{\varepsilon}_{sc}$ in a compact around $\Omega^{\varepsilon,c}$, let us define an appropriate  cutoff function.
\begin{definition}
	We fix a regular function $\chi\in C^{\infty}(\mathbb{R}^3; [0, 1])$, s.t. $\chi(\bx)=1$ for $|\bx|<1$, $\chi(\bx)=0$ for $|\bx|\geq 2$. Let  $\chi_a(\vec{x}):=\chi(a^{-1}\vec{x})$, for all $\vec{x}\in \mathbb{R}^3$, $a>0$. 
\end{definition}
Evidently, $\chi_a=1$ on $B_a(\vec{0})$ and vanishes on $B_{2a}^c(\vec{0})$. In what follows we will use the shortened notation: 
\begin{align*}
	B_{\rho}:=B_{\rho}(\boldsymbol{0}), \quad \rho>0,  \qquad B_{\rho_1, \rho_2}:=B_{\rho_2}\setminus \overline{B_{\rho_1}}, \quad \rho_2>\rho_1>0. 
\end{align*}
From the results of \cite{burq}, it is known that the operator-valued function $\mathbb{C}^+\ni s\mapsto \chi_b \mathcal{R}_D^{\re}(s)\chi_a\in \mathcal{L}(L^2(B_{2a}\cap \Omega^{\varepsilon,c}), L^2(B_{2b}\cap \Omega^{\varepsilon,c}))$ admits an analytic continuation in a certain region $W^{\re}$ of the complex plane that intersects $\mathbb{C}^{-}$. The size of this region is determined by the geometry of the obstacle. On the other hand, the function $s\mapsto \chi_b \mathcal{R}^{\re}_0(s)\chi_a\in \mathcal{L}(L^2(B_{2a}\cap \Omega^{\varepsilon,c}), L^2(B_{2b}\cap \Omega^{\varepsilon,c}))$ is entire. Those analyticity properties enable us to deform the integral contour of \eqref{eq:representation_formula_main} and to obtain estimates close to optimal for large times.

Let us introduce the notation: 
\begin{align*}
	\chi_{r,R}:=\chi_R(1-\chi_{r/2}), \quad \chi_{0,R}:=\chi_R, 
\end{align*}
{ so that }$\chi_{r,R}=1 \text{ on }B_R\setminus {B_r}$ and vanishes  for $\|\bx\|<r/2$ and $\|\bx\|>2R$. We will be particularly interested in the behaviour of the scattered field away from $\Omega^{\re}$:
\begin{align}
	\label{eq:ueps}
	\chi_{\rho_1, \rho_2}u^{\varepsilon}_{sc}(t)=\frac{1}{2\pi i}\int_{i\mathbb{R}+\mu}\frac{e^{st}}{(\mu_*-s)^p}\chi_{\rho_1, \rho_2}\left(\mathcal{R}^{\varepsilon}_D(s)-\mathcal{R}_0^{\re}(s)\right)\chi_{R_0}(su_0^{p}+v_0^{p})ds,
\end{align}
where $\rho_2>\rho_1>\re$ are given, and 
where we used the fact that $\operatorname{supp}\vec{u}_0^p\subseteq \operatorname{supp}\vec{u}_0\subset B(0,R_0)$, as argued in \eqref{eq:ic}. The key idea is the following:
\begin{itemize}
	\item we start with the expression \eqref{eq:ueps} and study the analyticity properties and $\varepsilon$-dependent bounds for the map 
$
		s\mapsto \chi_{\rho_1, \rho_2}\left(\mathcal{R}^{\varepsilon}_D(s)-\mathcal{R}_0^{\re}(s)\right)\chi_{R_0}.$
	\item using these results, we will show that the integration contour \eqref{eq:ueps} can be deformed in a way that will enable us to obtain the bounds on $\chi_{\rho_1, \rho_2}u^{\varepsilon}_{sc}(t). $
\end{itemize} 
In view of the above discussion, let us define the following operator, for $s\in \mathbb{C}^+$, 
\begin{align}
	\label{eq:ded}
	\mathcal{D}^{\re}_D(s):=\mathcal{R}^{\re}_D(s)-\mathcal{R}_0^{\re}(s), \quad \mathcal{D}^{\re}_D(s): \, L^2(\Omega^{\re,c})\rightarrow L^2(\Omega^{\re, c}). 
\end{align}
Our goal now is to obtain estimates on $
	s\mapsto \chi_{b}\mathcal{D}^{\re}_D(s)\chi_{a}$ 
explicit in $\re$, in certain operator norms. One could have argued that applying a simple scaling argument $\vec{x}\rightarrow \frac{\vec{x}}{\varepsilon}$ and using the results on $\mathcal{R}^{1}_D(s)$, $\mathcal{R}_0^1(s)$, available in the literature should yield the desired bounds. Unfortunately, this is not the case, since  such estimates require having bounds of the type $\chi_{b/\re} \mathcal{R}^{1}_D(\re s)\chi_{a/\re}$. At the same time, most available results in the literature for $\chi_b R^1_D(s) \chi_a$ are implicit in $b$ and $a$ (see, however, \cite{galkowski_spence_wunsch}), and thus we cannot pursue this path. Instead, we derive a Krein-type resolvent formula for $\mathcal{D}^{\re}_D$, cf. \cite{mantile_posilicano_sini}. 
\subsection{A Krein-like resolvent formula for $\mathcal{D}^{\re}_D$}
Let $s\in \mathbb{C}^+$,  $
	f\in L^2(\Omega^{\re,c}), \; p^{\re}:=\mathcal{D}^{\varepsilon}_D(s)f.$ 
Then, by definition, $p^{\re}\in H^2(\Omega^{\re,c})$ is a unique solution to
\begin{align}
	\label{eq:ald}
	\begin{split}
	(s^2-\Delta)p^{\re}=0 \text{ in }\Omega^{\re,c},\qquad 
	\gamma_0^{\re}p^{\re}=-\gamma_0^{\re}\mathcal{R}_0(s)f.
	\end{split}
\end{align}
Let us introduce the solution operator for the exterior Dirichlet problem:
\begin{align*}
\mathcal{N}^{\re}_D(s): \, H^{1/2}(\Gamma^{\re})\rightarrow H^1(\Omega^{\re,c}), \quad 	\mathcal{N}^{\varepsilon}_Dg:=v_g, 
\end{align*}
where $v_g\in H^1(\Omega^{\re,c})$ is a unique solution to 
\begin{align*}
&(s^2-\Delta) v_g=0 \text{ in }\Omega^{\re,c},\qquad 
\gamma_0^{\re} v_g=g.
\end{align*}
In this case, by definition, 
\begin{align}
	\label{eq:hatug}
p^{\re}=\mathcal{D}^{\varepsilon}_D(s)f=-\mathcal{N}^{\varepsilon}_D(s)\gamma_0^{\varepsilon}\mathcal{R}_0(s)f.
\end{align}
Our goal is to express $\mathcal{N}^{\varepsilon}_D(s)\gamma_0^{\varepsilon}\mathcal{R}_0(s)$ via   $\chi_{a\varepsilon}\mathcal{R}^{\re}_D(s)\chi_{b\varepsilon}$, with $a, b$ fixed, which will contain an important geometric information about the obstacle (and for which the estimates available in the literature can be easily adapted). We see immediately that the heterogeneity is concentrated on the boundary rather than in the right-hand side; on the other hand, to use $\mathcal{R}^{\varepsilon}_D(s)$, we need to have vanishing boundary conditions. To deal with this, we will use a well-chosen lifting of $\gamma_0^{\re}\mathcal{R}_0(s)f$, concentrated in the vicinity of $\Omega^{\varepsilon}$. 
\begin{figure}
	\centering
	\includegraphics[width=0.45\textwidth]{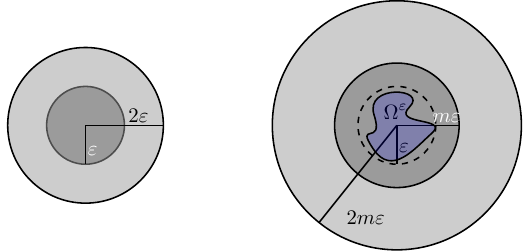}
	\caption{Left: support of $\chi_{\varepsilon}$ on the plane $z=0$. In dark gray we mark the sphere where $\chi_{\varepsilon}=1$, and in light gray where $\chi_{\varepsilon}$ varies between $0$ and $1$. Outside of the sphere of radius $2\varepsilon$, $\chi_{\varepsilon}=0$.
	 Right: support of $\chi_{m\varepsilon}$, $m\geq 2$, vs the obstacle $\Omega^{\varepsilon}$, $z=0$.}
 	\label{fig:illustration_to_Krein}
\end{figure}

\textit{Reducing the inhomogeneous Dirichlet problem to the homogeneous one.} Let us set $m$ as below and rewrite $p^{\re}$ as the following sum:
\begin{align}
	\label{eq:presc}
p^{\re}(s)=\chi_{m\varepsilon} p^{\re}_{sc}(s,z)+{r}^{\re}(s,z),\qquad m\geq 2, 	
\end{align}
 where $p^{\re}_{sc}(s,z)$ satisfies \eqref{eq:ald} with $s$ replaced by $z$ in the first equation, i.e. 
 \begin{align}
 	\label{eq:pre}
 	p^{\re}_{sc}(s,z)=-\mathcal{N}^{\varepsilon}_D(z)\gamma_0^{\varepsilon}\mathcal{R}_0(s)f.
 \end{align}
It appears that for $z>0$, one can estimate  $p^{\re}_{sc}(s,z)$ in a manner close to optimal (this will be discussed later). On the other hand, using that $(z^2-\Delta)p^{\re}_{sc,z}=0$ in $\Omega^{\re,c}$, we verify that  $r^{\re}$ satisfies
\begin{align}
	\label{eq:ald2}
	(s^2-\Delta){r}^{\re}={\left([\Delta, \chi_{m\varepsilon}]-(s^2-z^2)\chi_{m\varepsilon}\right)p^{\re}_{sc}(s,z)}\quad \text{ in }\Omega^{\re, c},
	\qquad \gamma_0^{\re}{r}^{\re}=0.
\end{align}
The equality on the trace follows by remarking that $\gamma_0^{\re}r^{\varepsilon}=\gamma_0^{\re}p^{\re}-\gamma_0^{\re}\chi_{m\varepsilon}p^{\re}_{sc}$, $\Omega^{\varepsilon}\subseteq B_{\varepsilon}$, and $\chi_{m\varepsilon}=1$ on $\partial\Omega^{\varepsilon}$ due to $m\geq 2$, see Figure \ref{fig:illustration_to_Krein}.  Let us now introduce a new differential operator
\begin{align*}
	L_{\delta}(s,z):=[\Delta, \chi_{\delta}]-(s^2-z^2)\chi_{\delta}=\left(\Delta\chi_{\delta}\right)+2\nabla\chi_{\delta}\nabla-(s^2-z^2)\chi_{\delta}.
\end{align*}
Thus, using the above notation and the definition of $\mathcal{R}^{\varepsilon}_D$, 
\begin{align}
	\label{eq:rRe}
	r^{\re}(s,z)&=\mathcal{R}^{\re}_D(s)	L_{m\varepsilon}(s,z)p^{\re}_{sc}(s,z)\equiv \mathcal{R}^{\re}_D(s)\chi_{2m\varepsilon}	L_{m\varepsilon}(s,z)p^{\re}_{sc}(s,z),
\end{align}
where we use the fact that $\operatorname{supp}L_{m\varepsilon}(s,z)\varphi\subseteq \operatorname{supp}\chi_{m\varepsilon}$ for all $\varphi$, and that $\chi_{2m\varepsilon}=1$ on $\operatorname{supp}\chi_{m\varepsilon}$, cf. Figure \ref{fig:illustration_to_Krein}. We have made appear the quantity $\mathcal{R}^{\re}_D(s)\chi_{2m\varepsilon}$, but our goal is to have   $\chi_{b\varepsilon}\mathcal{R}^{\re}_D(s)\chi_{2m\varepsilon}$.

\textit{A new expression for $r^{\varepsilon}$ involving only the truncated resolvent $\chi_{c\varepsilon}\mathcal{R}^{\re}_D\chi_{b\varepsilon}$. }
The key idea is to study the difference between $r^{\re}$ and $\chi_{n\re}r^{\re}$, defined below:
\begin{align}
	\label{eq:assumption_n}
r^{\varepsilon}_{\operatorname{ff}}:=	r^{\varepsilon}-r^{\re}_{\operatorname{nf}}, \quad r^{\re}_{\operatorname{nf}}:=\chi_{n\varepsilon}r^{\varepsilon}, \quad n\geq 2.
\end{align}
With \eqref{eq:rRe}, we see that 
\begin{align}
	\label{eq:nfe}
r^{\varepsilon}_{\operatorname{nf}}= \chi_{n\varepsilon}\mathcal{R}^{\varepsilon}_D(s)\chi_{2m\varepsilon}L_{m\varepsilon}(s,z)p^{\varepsilon}(s,z),
\end{align}
where now the resolvent $\mathcal{R}_D^{\re}(s)$ appears surrounded with two $\varepsilon$-dependent truncations. 

On the other hand, we remark that  $r^{\varepsilon}_{\operatorname{ff}}=(1-\chi_{n\varepsilon})r^{\varepsilon}$, and thus $\left.r^{\varepsilon}_{\operatorname{ff}}\right|_{B_{n\varepsilon}}=0$; since $\Omega^{\varepsilon}\subseteq B_{\varepsilon}\subsetneq B_{n\varepsilon}$, $r^{\varepsilon}_{\operatorname{ff}}\in H^2_0(\Omega^{\varepsilon,c})$. Therefore, its continuation by zero inside $\Omega^{\re,c}$ is of regularity $\mathbb{E}^{\re}r^{\re}_{\operatorname{ff}}\in H^2(\mathbb{R}^3)$, i.e.  $\mathbb{E}^{\re}r^{\re}_{\operatorname{ff}}\in D(\Delta_0)$. 
The key idea is then to express $r^{\re}_{\operatorname{ff}}$ using the cut-off free resolvent $\mathcal{R}_0^{\re}(s)=\mathbb{R}^{\re}\mathcal{R}_0(s)\mathbb{E}^{\re}$. 

We rewrite 
\begin{align}
	\label{eq:auxiliary_r}
	(s^2-\Delta_0)\mathbb{E}^{\re}r^{\re}_{\operatorname{ff}}&=\mathbb{E}^{\re}(s^2-\Delta_0)\mathbb{E}^{\re}r^{\re}_{\operatorname{ff}}=\mathbb{E}^{\re}(s^2-\Delta)r^{\re}_{\operatorname{ff}}.
\end{align}
From $r^{\re}_{\operatorname{ff}}=(1-\chi_{n\re})r^{\re}$ it follows that (remark that $\Delta$ below is a differentiation operator):
\begin{align}
	\label{eq:rff}
	(s^2-\Delta)r^{\re}_{\operatorname{ff}}&=(1-\chi_{n\re})(s^2-\Delta)r^{\re}+[\Delta, \chi_{n\re}]r^{\re} \text{ in }\Omega^{\re,c}.
\end{align}
We simplify the right-hand side; we choose $n=2m$, and use \eqref{eq:ald2}: 
\begin{align*}
(1-\chi_{2m\re})(s^2-\Delta)r^{\re}=(1-\chi_{2m\re})([\Delta, \chi_{m\re}]-(s^2-z^2)\chi_{m\re})p^{\re}_{sc}(s,z)=0,
\end{align*}
since $1-\chi_{2m\re}=0$ on $\overline{B_{2m\re}}\supseteq\operatorname{supp}\chi_{m\re}$. This allows to rewrite \eqref{eq:rff} as follows:
 \begin{align*}
 	(s^2-\Delta)r^{\re}_{\operatorname{ff}}&=[\Delta, \chi_{2m\re}]r^{\re} \text{ in }\Omega^{\re,c}.
 \end{align*}
Next, we replace in the above $r^{\re}$ by \eqref{eq:rRe}; moreover, we recall the identity \eqref{eq:auxiliary_r}:
 \begin{align*}
	(s^2-\Delta_0)\mathbb{E}^{\re}r^{\re}_{\operatorname{ff}}&=\mathbb{E}^{\re}[\Delta, \chi_{2m\re}]\mathcal{R}^{\re}_D(s)\chi_{2m\re}L_{m\re}(s,z)p^{\re}_{sc}(s,z).
\end{align*}
Therefore, 
\begin{align*}
r^{\re}_{\operatorname{ff}}=\mathcal{R}^{\re}_0(s)[\Delta, \chi_{2m\re}]\mathcal{R}^{\re}_D(s)\chi_{2m\re}L_{m\re}(s,z)p^{\re}_{sc}(s,z).
\end{align*}
Combining the above expression with \eqref{eq:nfe} into \eqref{eq:assumption_n} yields
\begin{align*}
r^{\varepsilon}(s,z)&=
\left(\mathcal{R}_0^{\re}(s)[\Delta, \chi_{2m\re}]+\chi_{2m\re}\right)\mathcal{R}^{\re}_D(s)\chi_{2m\re}L_{m\re}(s,z)p^{\re}_{sc}(s,z)\\
&=\left(\mathcal{R}_0^{\re}(s)[\Delta, \chi_{2m\re}]+\chi_{2m\re}\right)\chi_{4m\re}\mathcal{R}^{\re}_D(s)\chi_{2m\re}L_{m\re}(s,z)p^{\re}_{sc}(s,z),
\end{align*}
where we used that $\chi_{4m\re}=1$ on $\operatorname{supp}\chi_{2m\re}$. Replacing in \eqref{eq:presc} $r^{\re}(s,z)$ by the rhs of the above and $p^{\re}_{sc}(s,z)$ by the rhs of \eqref{eq:pre}, and recalling the definition \eqref{eq:hatug} yields
\begin{proposition}[Krein-like formula for the resolvent]
	\label{prop:De}
	For all $s,z\in \mathbb{C}^+$, $0<\re<1$, $m\in \mathbb{N}: \, m\geq 2$, the following identity holds true
	\begin{align*}
	&\mathcal{D}^{\re}_D(s)=\mathcal{R}^{\re}_D(s)-\mathcal{R}_0^{\re}(s)=-\left(\chi_{m\re}+Q_{\re}(s,z)\right)\mathcal{N}^{\re}_D(z)\gamma_0^{\re}\mathcal{R}_0(s), \\ 
	&Q_{\re}(s,z)=-(\mathcal{R}_0^{\varepsilon}(s)[\Delta,\chi_{2m\varepsilon}]+\chi_{2m\re})\chi_{4m\varepsilon}\mathcal{R}^{\re}_D(s)\chi_{2m\re}L_{m\re}(s,z),\\
	&L_{m\re}(s,z)=[\Delta, \chi_{m\varepsilon}]-(s^2-z^2)\chi_{m\varepsilon}.
	\end{align*}
\end{proposition}
In what follows, we will estimate all the terms occurring in the above expression for $\mathcal{D}^{\re}_D(s)$. While we have defined the above expression for $s\in \mathbb{C}^+$, we will argue that it is possible to continue it analytically in a certain topology for a range of $s\in \mathbb{C}^{-}$, and derive estimates for this, extended, range of $s$. 
\subsection{Behavior of $\mathcal{D}^{\re}_D$ on $\mathbb{C}$}
Let $
	|\Re s|_{-}:=\max(0, -\Re s).
$
We then have 
\begin{proposition}
	\label{prop:bounds}
Let $0<\rho<\infty$, $0<\re<1$. The operator-valued function 
	\begin{align*}
		s\mapsto \mathcal{D}^{\re}_D(s)\in \mathcal{L}(L^2(\Omega^{\re,c}), H^1_0(\Omega^{\re,c})\cap H^2(\Omega^{\re,c})),
	\end{align*}
is analytic on $\mathbb{C}^+$. Viewed as the function
$
	s\mapsto \chi_{\rho}\mathcal{D}^{\re}_D(s)\chi_{\rho}\in \mathcal{L}(L^2(\Omega^{\re,c}), L^2(\Omega^{\re,c}))$, 
it admits an analytic extension into the following region of the complex plane:
\begin{align*}
\mathscr{D}^{\re}=\{s\in \mathbb{C}:\,	\Re s>-c_{\Omega}\re^{-1}\},
\end{align*}
where $c_{\Omega}$ is a strictly positive constant that is independent of $\re$, $\rho$, but depends on the geometry of the scatterer $\Omega$. 

Moreover, there exist $\re_0, \, C>0$, s.t. for all $0<\re<\re_0$, $s\in \mathbb{C}_{\Omega}$,  $\mathcal{D}^{\re}_D$ satisfies the following bounds. In the far-field, for $R>r>1$, 
\begin{align*}
	\|\chi_{r,R}\mathcal{D}^{\re}_D(s)\chi_{R_0}\|_{L^2(\Omega^{\re,c})\rightarrow L^2(\Omega^{\re,c})}\leq CC_{r,R} c_{\mathcal{O}_0}\re \mathrm{e}^{|\Re s|_{-}(R_0+R+\widetilde{c}_{\Omega}^+\re)}(1+|s\re|)^3,
\end{align*}
where $c_{\mathcal{O}_0}=R_0^{3/2}\max(1, r_0^{-2})$, $C_{r,R}=R^{3/2}r^{-1}$, and $\widetilde{c}_{\Omega}^+>0$ depends solely on the geometry of the scatterer $\Omega$. 
\end{proposition}
	The proof of this proposition can be found in Section \ref{sec:proof}; it is based on several preliminary results. In particular, we start with the expression for $\mathcal{D}^{\re}_D$ of Proposition \ref{prop:De}. 
Bounding norm of $\chi_{\rho}\mathcal{D}^{\re}_D(s)\chi_{R_0}$ amounts to fixing $z\in \mathbb{C}^+$, $m\geq 2$, and bounding norms of the following operators (possibly localized): $\gamma_0^{\re}\mathcal{R}_0(s)$ (Proposition \ref{prop:nz_aux}), $\mathcal{N}_D^{\re}(z)$ (Theorem \ref{prop:nz}),  $\chi_{4m\re}\mathcal{R}^{\re}_D(s)\chi_{2m\re}$ (Co\-rollary \ref{cor:analyticity}) and $\mathcal{R}_0^{\re}(s)$ (Proposition \ref{prop:free_resolvent}). This is a subject of the following sections. 
\subsubsection{Estimates involving the free resolvent}
In this section we summarize all the estimates related to the free resolvent. In our statements we will often resort to the fact that the free (cut-off) resolvent $s\mapsto \chi_R\mathcal{R}_0(s)\chi_R\in \mathcal{L}(L^2(\mathbb{R}^3), H^j(\mathbb{R}^3))$, $j=0,1,2$, is an entire function of $s\in \mathbb{C}$, see e.g. \cite[Theorem 3.1]{dyatlov_zworski}. We start by studying $L^2$-bounds in $\Omega^{\re,c}$.
\begin{proposition}[Bounds in $\Omega^{\re,c}$]
	\label{prop:free_resolvent}
	Let $c>0$, $R>r>c\re$. Then for all $v\in L^2(\mathbb{R}^3)$, s.t. $\operatorname{supp}v\subseteq \overline{B_{c\re}}$,  
	\begin{align*}
		\|\mathcal{R}_0(s)v\|_{L^2(B_{r,R})}\lesssim \mathrm{e}^{|\Re s|_{-}(R+c\re)} \frac{R^{3/2}}{r-c\varepsilon}(c\varepsilon)^{3/2}\|v\|_{L^2(B_{c\re})}, \quad s\in \mathbb{C}.
	\end{align*} 
\end{proposition}
\begin{proof}
	We start with the explicit identity
	\begin{align}
		\label{eq:identR0}
		\left(\mathcal{R}_0(s)v\right)(\bx)=\int_{B_{c\re}}\frac{\mathrm{e}^{-s\|\bx-\by\|}}{4\pi\|\bx-\by\|}v(\by)d\by,
	\end{align}
	which we then estimate in two different ways. First of all, using a very rough $L^{\infty}$-bound on the integral kernel, and next the Cauchy-Schwarz inequality, we obtain that
	\begin{align*}
		\|\mathcal{R}_0(s)v\|_{L^{\infty}(B_{r,R})}\leq \frac{\mathrm{e}^{|\Re s|_{-}(R+c\re)}}{4\pi (r-c\re)}\|v\|_{L^1(B_{c\re})}\lesssim \frac{\mathrm{e}^{|\Re s|_{-}(R+c\re)}}{ (r-c\re)}(c\re)^{3/2}\|v\|_{L^2(B_{c\re})}.
	\end{align*}  
The desired far-field bound follows from    $\|\mathcal{R}_0(s)v\|_{L^{2}(B_{r,R})}\lesssim R^{3/2}\|\mathcal{R}_0(s)v\|_{L^{\infty}(B_{r,R})}$.
	
\end{proof}
Next, we will need bounds on $\gamma_0^{\re}\mathcal{R}_0(s)f$. We will need them written in a rather peculiar manner, adapted to the statement of Theorem \ref{prop:nz} that will follow later. 
Let us introduce auxiliary projection operators defined onto the space of constants and functions orthogonal to constants, namely,
\begin{align*}\mathbb{S}_0 := \operatorname{span}\{1\}, \quad H^{1/2}_{\perp}(\Gamma^{\re}) := \{ \psi \in H^{1/2}(\Gamma^{\re}) : (\psi, 1)_{L^2(\Gamma^{\re})} = 0 \}.
\end{align*}
Let $\mathbb{P}_0^{\re}$ be an $L^2$-orthogonal projection onto the space $\mathbb{S}_0$, and 
$\mathbb{P}_{\perp}^{\re} = \mathbb{I} - \mathbb{P}_0^{\re}$ onto the space $H^{1/2} _{\perp}(\Gamma^{\re})$. The proposition below shows that $\mathbb{P}_0^{\re}\gamma_0^{\re}\mathcal{R}_0(s)f$ scales as $O(\re)$, while $\mathbb{P}_{\perp}^{\re}\gamma_0^{\re}\mathcal{R}_0(s)f$ as $O(\re^{3/2})$. 
\begin{proposition}
	\label{prop:nz_aux}
	Assume that $\operatorname{supp}f \subseteq \overline{\mathcal{O}}_0=\overline{B_{r_0,R_0}}$, where $R_0>r_0>1$ are like in \eqref{eq:ic}. 
Then, for $0<\re\leq 1$, $s\in \mathbb{C}$, 
\begin{align*}
	&\|\mathbb{P}_0^{\re}\gamma_0^{\re}\mathcal{R}_0(s)f\|_{L^2(\Gamma^{\re})}\lesssim \re R_0^{3/2}(r_0-\re)^{-1}\mathrm{e}^{|\Re s|_{-}(R_0+\re)}\|f\|_{L^2(\mathcal{O}_0)},\\
	&\|\mathbb{P}_{\perp}^{\re}\gamma_0^{\re}\mathcal{R}_0(s)f\|_{H^{1/2}(\Gamma^{\re})}\lesssim \re^{3/2}R_0^{3/2}(r_0-\re)^{-1}\mathrm{e}^{|\Re s|_{-}(R_0+\re)}(|s|+(r_0-\re)^{-1})\|f\|_{L^2(\mathcal{O}_0)}.		
\end{align*}	
\end{proposition}
This proposition is immediate from Lemmas \ref{lem:rz0}, \ref{lem:rz}, proven below, thus we leave its proof to the reader.  
\begin{lemma}
	\label{lem:rz0}
	Let $u\in W^{1,\infty}(\Omega^{\re})$. Then 
	\begin{align*}
		&\|\mathbb{P}_0^{\re}\gamma_0^{\re}u\|_{L^2(\Gamma^{\re})}\lesssim \varepsilon\|u\|_{L^{\infty}(\Omega^{\re})},\qquad 
		\|\mathbb{P}^{\re}_{\perp}\gamma_0^{\re}u\|_{H^{1/2}(\Gamma^{\re})}\lesssim \varepsilon^{3/2}\|\nabla u\|_{L^{\infty}(B_{\re})}.
	\end{align*}
\end{lemma}
\begin{proof}
	To prove the first inequality, we use an explicit characterization of the projection operator $\mathbb{P}_0^{\re}$:
	\begin{align*}
		|\mathbb{P}_0^{\re} \gamma_0^{\re}u| = \frac{|(\gamma_0^{\re}u, 1)_{L^2(\Gamma^{\re})}|}{\| 1 \|^2_{ L^2(\Gamma^{\re})}} \leq \frac{\| \gamma_0^{\re}u \|_{L^2(\Gamma^{\re})}}{\| 1 \|_{L^2(\Gamma^{\re})}} \leq \| u \|_{L^{\infty}(\Omega^{\re})}.
	\end{align*}
	The desired inequality follows from the above and $\| 1 \|_{L^2(\Gamma^{\re})} \lesssim \re$. 
	
	To show the second bound in the statement, we use \cite[Proposition 4.3]{mk_as}:  
	\begin{align*}
		\| \mathbb{P}_{\perp}^{\re} \gamma^{\re}_0 u \|_{H^{1/2}(\Gamma^{\re})} \leq C |\mathbb{P}_{\perp}^{\re}\gamma^{\re}_0 u|_{H^{1/2}(\Gamma^{\re})}=C|\gamma^{\re}_0 u|_{H^{1/2}(\Gamma^{\re})}, \quad \forall 0<\re\leq 1. 
	\end{align*}
	Using the definition of the Sobolev-Slobodeckij seminorm and the mean-value theorem, 
	\begin{align*}
|\gamma^{\re}_0 u|_{H^{1/2}(\Gamma^{\re})}&= \iint_{\Gamma^{\re} \times \Gamma^{\re}} \frac{|u(\bx) - u(\by)|^2}{\| \bx - \by \|^3} d\Gamma_{\bx} d\Gamma_{\by} \leq  \re^3 \iint_{\Gamma^1 \times \Gamma^1} \frac{  \| \nabla u \|^2_{L^{\infty}({B_{\re}})}}{\| \bx - \by \|} d\Gamma_{\bx} d\Gamma_{\by}.
	\end{align*}
The Lebesgue integral above is finite as a weakly-singular integral. 
\end{proof}
\begin{lemma}
	\label{lem:rz}
	Assume that $f$ is like in Proposition \ref{prop:nz_aux}. Then, for all $s\in \mathbb{C}$, 
	\begin{align*}
	&\|\mathcal{R}_0(s)f\|_{L^{\infty}(B_{\re})}\lesssim R_0^{3/2}(r_0-\re)^{-1}\mathrm{e}^{|\Re s|_{-}(R_0+\re)}\|f\|_{L^2(\mathcal{O}_0)},\\
	&\|\nabla \mathcal{R}_0(s)f\|_{L^{\infty}(B_{\re})}\lesssim R_0^{3/2} (|s|+(r_0-\re)^{-1}) (r_0-\re)^{-1}\mathrm{e}^{|\Re s|_{-}(R_0+\re)}\|f\|_{L^2(\mathcal{O}_0)}.
\end{align*}
\end{lemma}
\begin{proof}
	We show one bound only, since the remaining one is obtained in a similar manner. By the Cauchy-Schwarz inequality,
\begin{align*}
	\|\mathcal{R}_0(s)f\|_{L^{\infty}(B_{\re})}&=\left\|\int_{\mathcal{O}_0}\frac{\mathrm{e}^{- s\|\vec{x}-\vec{y}\|}}{4\pi\|\vec{x}-\vec{y}\|}f(\vec{y})d\vec{y}\right\|_{L^{\infty}(B_{\re})}\\
	&\lesssim \frac{R_0^{3/2}\mathrm{e}^{|\Re s|_{-}(R_0+\re)}}{\operatorname{dist}(\operatorname{supp}f, B_{\re})}\|f\|_{L^2}\lesssim \frac{R_0^{3/2}\mathrm{e}^{|\Re s|_{-}(R_0+\re)}}{r_0-\re}\|f\|_{L^2}. 
\end{align*}
\end{proof}
\subsubsection{Estimates on $\mathcal{N}^{\varepsilon}_D(z)$}
We are interested in the estimates on the operator $\mathcal{N}_D^{\re}(z)$, for the case when $z>0$, where such  bounds are 'easy' to obtain. 
\begin{theorem}
	\label{prop:nz}
	Let $\nu>0$ and $\re\in (0,1]$. Then, for all $\varphi\in H^{1/2}(\Gamma^{\re})$: 
	\begin{align*}
		&\|\mathcal{N}_D^{\varepsilon}(\nu)\mathbb{P}_0^{\re}\varphi\|_{L^2(\Omega^{\re,c})}\lesssim \varepsilon^{-1/2}\nu^{-1}\sqrt{\max(1,\nu\varepsilon)}\|\mathbb{P}_0^{\re}\varphi\|_{H^{1/2}(\Gamma^{\varepsilon})},\\
		&\|\mathcal{N}_D^{\varepsilon}(\nu)\mathbb{P}_{\perp}^{\re}\varphi\|_{L^2(\Omega^{\re,c})}\lesssim \nu^{-1}\sqrt{\max(1,\nu\varepsilon)}\|\mathbb{P}_{\perp}^{\re}\varphi\|_{H^{1/2}(\Gamma^{\varepsilon})},\\
		&| \mathcal{N}_D^{\varepsilon}(\nu)\mathbb{P}_0^{\re}\varphi|_{H^1(\Omega^{\re,c})}\lesssim \varepsilon^{-1/2}\sqrt{\max(1,\nu\varepsilon)}\|\mathbb{P}_0^{\re}\varphi\|_{H^{1/2}(\Gamma^{\varepsilon})},\\
		&|\mathcal{N}_D^{\varepsilon}(\nu)\mathbb{P}_{\perp}^{\re}\varphi|_{H^1(\Omega^{\re,c})}\lesssim \sqrt{\max(1,\nu\varepsilon)}\|\mathbb{P}_{\perp}^{\re}\varphi\|_{H^{1/2}(\Gamma^{\varepsilon})}.
	\end{align*}
\end{theorem}
\begin{proof}
	By the lifting lemma \cite[Lemma C.1.]{mk_as}, there exists $C>0$ such that for all $\xi \in H^{1/2}(\Gamma^{\re})$, $\re \in (0, 1]$, the unique solution $u \in H^1 (\Omega^{\re, c})$ to
	\begin{align*}
		- \Delta u + \nu^2 u = 0 ~\text{in}~\Omega^{\re, c}, \quad \gamma^{\re}_0 u = \xi,
	\end{align*}
i.e. $u=\mathcal{N}^{\re}_D(\nu)\xi$ satisfies the following bound:
	\begin{align*}
		\nu^2 \| u \|^2_{L^2(\Omega^{\re, c})} + |u|^2_{H^1(\Omega^{\re, c})} \leq C \max(1, \nu \re) (\re^{-1} \| \mathbb{P}_0^{\re} \xi \|^2_{H^{1/2}(\Gamma^{\re})} +\| \mathbb{P}_{\perp}^{\re} \xi \|^2_{H^{1/2}(\Gamma^{\re})}).
	\end{align*}
The desired result follows by taking $\xi=\mathbb{P}_0^{\re}\varphi$, $\mathbb{P}_{\perp}^{\re}\varphi$. 
\end{proof}
\subsubsection{Bounds on the cut-off resolvent $\chi_{a\re}\mathcal{R}^{\re}_D(s)\chi_{b\re}$}	
The simple scaling result that allows to obtain a suitable estimate on $\mathcal{R}^{\re}_D(s)$ from $\mathcal{R}^1_D(s)$ is formulated below. 
\begin{proposition}
	\label{prop:scaling}
	Given $R, r>0$, assume that the cut-off resolvent $\chi_R\mathcal{R}_D^1(s) \chi_r$ satisfies the following:
	\begin{itemize}
		\item the function $s\mapsto \chi_R\mathcal{R}_D^1(s) \chi_r$ is $\mathcal{L}\left(L^2(\Omega^{1,c}), H^{1}(\Omega^{1,c})\right)-$analytic in the region 
		\begin{align*}
			\mathscr{D}:=\{s\in \mathbb{C}: \, \Re s>-\varphi(|\Im s|)\},
		\end{align*}
		where $\varphi: \, \mathbb{R}_{\geq 0}\rightarrow \mathbb{R}_{>0}$ is independent of $R, r>0$. 
		\item there exist positive functions $M_{r,R}^j$, $j=0,1$, s.t. for all $s\in \mathscr{D}$, it holds that 
		\begin{align*}
			\|	\chi_R\mathcal{R}^1_D(s)\chi_r\|_{L^2\rightarrow H^{j}}\leq M_{r, R}^j(|s|).
		\end{align*}
	\end{itemize}
	Then for all $0<\re\leq 1$,  $
		s\mapsto 	\chi_{\re R}R^{\varepsilon}_D(s)\chi_{\re r}$ 
	satisfies  
	\begin{itemize}
		\item  the function $s\mapsto \chi_{\re R}\mathcal{R}_D^{\re}(s) \chi_{\re r}$ is $\mathcal{L}\left(L^2(\Omega^{\re,c}), H^1(\Omega^{\re,c})\right)-$analytic in  
		\begin{align*}
			\mathscr{D}^{\re}:=\{s\in \mathbb{C}: \, \Re s>-\re^{-1} \varphi(\re|\Im s|)\}.
		\end{align*}
		\item it holds that, with $j=0,1$,  
		\begin{align*}
			&\|	\chi_{\re R}\mathcal{R}^{\re}_D(s)\chi_{\re r}\|_{L^2\rightarrow H^j}\leq \re^{2-j} M_{r, R}^j(|\re s|).
		\end{align*}
	\end{itemize}
\end{proposition}
\begin{proof}
This is a simple scaling argument. Let $f\in L^2(\Omega^{\re,c})$. By definition,  $v^{\re}=\mathcal{R}^{\re}_D(s)\chi_{\varepsilon r} f$ is a solution to 
	\begin{align*}
		&(s^2-\Delta)v^{\re}=\chi_{\varepsilon r}f \quad \text{ in }\Omega^{\varepsilon,c},\\
		&\gamma_0^{\re}v^{\re}=0.
	\end{align*}
	Next, we introduce a dilation 
	\begin{align*}
		\mathcal{U}^{\re}: \, L^2(\Omega^{\re,c})\rightarrow L^2(\Omega^{1,c}), \quad \left(\mathcal{U}^{\re}v\right)(\bx)=\re^{3/2}v(\re \bx), \quad \bx\in \Omega^{1,c},
	\end{align*}
	where the additional scaling factor is chosen so that $\|\mathcal{U}^{\re}v\|_{L^2(\Omega^c)}=\|v\|_{L^2(\Omega^{\re,c})}$. With the parallelogram identity, this implies that $\mathcal{U}^{\re}$ is a unitary operator. Evidently, 
	\begin{align}
		\label{eq:identities}
		\mathcal{U}^{\re}\Delta=\re^{-2}\Delta \mathcal{U}^{\re}, \text{ and }\mathcal{U}^{\re}(\chi_{\rho \re}v)=\chi_{\rho}\mathcal{U}^{\re}v, \qquad \text{ so that }\\
		\nonumber
		\mathcal{U}^{\re}(s^2-\Delta v^{\re})=\mathcal{U}^{\re}\chi_{\varepsilon r}f \text{ in }\Omega^{\re,c}\iff ((s\varepsilon)^2-\Delta)\mathcal{U}^{\re}v^{\re}=\varepsilon^2 \chi_r(\mathcal{U}^{\re}f)  \quad \text{ in }\Omega^{1,c}.
	\end{align}
	Therefore, 
	\begin{align*}
		\mathcal{U}^{\re}v^{\re}=\varepsilon^2 \mathcal{R}^1_D(\varepsilon s)\chi_r \mathcal{U}^{\re}f\implies \chi_R \mathcal{U}^{\re}v^{\re}=\varepsilon^2 \chi_R \mathcal{R}^1_D(\varepsilon s)\chi_r \mathcal{U}^{\re}f,
	\end{align*}
	which, combined with \eqref{eq:identities}, yields the following simple identity: 
	\begin{align*}
		\chi_{R\re}v^{\re}=\re^2(\mathcal{U}^{\re})^{-1}\chi_R \mathcal{R}^1_D(\varepsilon s)\chi_r \mathcal{U}^{\re}f. 
	\end{align*}
	Thus, 
	\begin{align*}
		\chi_{R\re}\mathcal{R}^{\re}_D(s)\chi_{r\re}=\re^2(\mathcal{U}^{\re})^{-1}\chi_R \mathcal{R}^1_D(\varepsilon s)\chi_r \mathcal{U}^{\re}.
	\end{align*}
The analyticity of 	$\chi_{R\re}\mathcal{R}^{\re}_D(s)\chi_{r\re}$ follows from the above identity. As for the bounds, since $\mathcal{U}^{\re}$ is unitary, and using the assumption of the proposition, we conclude that 
\begin{align*}
	\|\chi_{R\re}\mathcal{R}^{\re}_D(s)\chi_{r\re}\|_{\mathcal{L}(L^2(\Omega^{\re,c}), L^2(\Omega^{\re,c}))}\leq \varepsilon^2 M_{r,R}^0(|\re s|).
\end{align*}
The case for $j=1$ follows similarly, by remarking that $\nabla (\mathcal{U}^{\re})^{-1}=\re^{-1}(\mathcal{U}^{\re})^{-1}\nabla$
\begin{align*}
\nabla \chi_{R\re}\mathcal{R}^{\re}_D(s)\chi_{r\re}=\re(\mathcal{U}^{\re})^{-1}\nabla (\chi_R\mathcal{R}^1_D(\re s)\chi_r\mathcal{U}^{\re}),
\end{align*}
and using that $\mathcal{U}^{\re}$ is a unitary operator. 
\end{proof}
The above result allows to relate the norm of the cut-off resolvent $\mathcal{R}^{\re}_D(s)$ in the $\re$-vicinity of $\Omega^{\re}$ to the norm of the unscaled resolvent $\mathcal{R}^1_D(s)$ in the $O(1)$-vicinity $\Omega^1$. Bounds on $\mathcal{R}^1_D(s)$ are readily available in the existing literature, see the result below.
\begin{theorem}
		There exist $c_{\Omega}>0, \, c^+_{\Omega}>0$, s.t. for all $R, r>0$, the cut-off resolvent $\chi_R R^1_D(s)\chi_r\in \mathcal{L}(L^2(\Omega^c), H^{1}(\Omega^c))$, is an analytic operator-valued function on
		\begin{align*}
			\mathscr{D}=\{s\in \mathbb{C}: \, \Re s>-c_{\Omega}\}, 
		\end{align*} 
	and satisfies the following bound:
		\begin{align*}
			&\|\chi_{R}(s^2-\Delta_D^1)^{-1}\chi_r\|_{L^2(\Omega^c)\rightarrow H^{\ell}(\Omega^c)}\lesssim C_{r,R}(1+|s|)^{-1+\ell}\mathrm{e}^{c_{\Omega}^+|\Re s|_{-}}, \quad \ell=0,1.
\end{align*}
	\end{theorem}
	\begin{remark}
	In \cite{ralston} it was shown that $s\mapsto \chi_R R^1_D(s)\chi_r$ is analytic for $\Re s>-2(\operatorname{diam}\Omega)^{-1}$, see the related discussion in \cite{hintz_zworski}.
	\end{remark}
\begin{proof}
First of all, we remark that the analyticity of the resolvent in the required region follows by combining the following facts:
\begin{itemize}
	\item  the cut-off resolvent is a meromorphic $\mathcal{L}(L^2(\Omega^c), D(\Delta_D))$-function on $\mathbb{C}$, see \cite[Theorem 4.4]{dyatlov_zworski}; 
	\item analyticity of the cut-off resolvent in the vicinity of $0$ \cite[Theorem 4.19]{dyatlov_zworski}; 
	\item the statement of  \cite[Theorem 4.43]{dyatlov_zworski} for $|s|\gg 1$.
\end{itemize}
To apply Theorem 4.43 of \cite{dyatlov_zworski}, we need, however, to argue that star-shaped obstacles are \textit{black-box non-trapping} in the sense of Definition 4.42 \cite{dyatlov_zworski}. More precisely, we need to verify that   $t\mapsto \chi_R\frac{\sin (\sqrt{-\Delta_D}t)}{\sqrt{-\Delta_D}}\chi_r\in C^{\infty}((T_{r,R}; +\infty); \mathcal{L}(L^2(\Omega^{c}); H^1_0(\Omega^c)\cap H^2(\Omega^c)))$ for star-shaped obstacles. As pointed out in Remark 4 after Theorem 4.43, for the validity of Theorem 4.43 it is sufficient that, for some $k\geq 0$, $$t\mapsto \chi_R\frac{\sin (\sqrt{-\Delta_D}t)}{\sqrt{-\Delta_D}}\chi_r\in C^{k}((T_{r,R}; +\infty); \mathcal{L}(L^2(\Omega^{c}); H^1_0(\Omega^c)\cap H^2(\Omega^c))).$$

This latter estimate, for arbitrary $k\geq 0$, was proven in \cite[Theorem 1.4]{melrose} for geometrically non-trapping obstacles, and obstacles satisfying Assumption \ref{assumption:star_shaped} are indeed geometrically non-trapping, see \cite[Proposition 13.3.1]{petkov_stoyanov}.

The respective bound for $\ell=0$ on the resolvent follows from  \cite[Teorem 4.43]{dyatlov_zworski} for $|s|\gg 1$, and for $|s|<\operatorname{const}$ it is immediate from the analyticity of $s\mapsto\chi_R R^1_D(s)\chi_r\in \mathcal{L}(L^2(\Omega^c), H^{\ell}(\Omega^c))$ in the $\mathbb{C}^{-}$-vicinity of the imaginary axis.

To obtain a bound for $\ell=1$, we make use of Remark 1 after Theorem 4.43 of \cite{dyatlov_zworski}; the corresponding estimate is stated as estimate (4.6.44) in \cite{dyatlov_zworski}. \footnote{Remark however a small typo: in (4.6.44) $\alpha\rightarrow \alpha/2$ everywhere but in $\langle \lambda\rangle^{\alpha/2}$.}

\end{proof}
Applying the above bound in Proposition \ref{prop:scaling} yields the desired estimate. 
\begin{corollary}
	\label{cor:analyticity}
There exist $c_{\Omega}, \, c^+_{\Omega}>0$, s.t. for $\ell=0,1$, $R,r>0$, $0<\re\leq 1$, the cut-off resolvent $s\mapsto \chi_{\re R}\mathcal{R}_D^{\re}(s) \chi_{\re r}$ is $\mathcal{L}\left(L^2(\Omega^{\re,c}), H^{\ell}(\Omega^{\re,c})\right)-$analytic in the region $
			\mathscr{D}^{\re}=\{s\in \mathbb{C}: \, \Re s>-\re^{-1} c_{\Omega}\},
$
and satisfies the following bound:
		\begin{align*}
			&\|	\chi_{\re R}\mathcal{R}^{\re}_D(s)\chi_{\re r}\|_{L^2\rightarrow H^{\ell}}\lesssim \re^{2-\ell} (1+|\re s|)^{-(1-\ell)}\mathrm{e}^{\re c_{\Omega}^+| \Re s|_{-}}.
		\end{align*}
\end{corollary}
\subsubsection{Proof of Proposition \ref{prop:bounds}}
\label{sec:proof}
We use Proposition \ref{prop:De} with $m=2$: 
\begin{align*}
	&\mathcal{D}^{\re}_D(s)=-\left(\chi_{2\re}+Q_{\re}(s,z)\right)\mathcal{N}^{\re}_D(z)\gamma_0^{\re}\mathcal{R}_0(s), \\ 
	&Q_{\re}(s,z)=-(\mathcal{R}_0^{\varepsilon}(s)[\Delta,\chi_{4\varepsilon}]+\chi_{4\re})\chi_{8\varepsilon}\mathcal{R}^{\re}_D(s)\chi_{4\re}\left([\Delta, \chi_{2\varepsilon}]-(s^2-z^2)\chi_{2\varepsilon}\right).
\end{align*}
Since $[\Delta, \chi]=2\nabla \chi\cdot \nabla +\Delta \chi\cdot $, careful examination of the above expression and 
analyticity of the resolvents  $\mathcal{R}_0(s)$, $\mathcal{R}_0^{\re}(s)$, $\mathcal{R}_D^{\re}(s)$ on $\mathbb{C}^+$ as functions from $L^2$ into $H^1$ yields the conclusion about the analyticity of $s\mapsto \mathcal{D}^{\re}_D(s)\in \mathcal{L}(L^2(\Omega^{\re,c}), L^2(\Omega^{\re,c}))$. 

Next, to argue about the analyticity of $s\mapsto\chi_{\rho} \mathcal{D}^{\re}_D(s)\chi_{\rho}$, we rewrite 
\begin{align*}
	\chi_{\rho} \mathcal{D}^{\re}_D(s)\chi_{\rho}=	-\left(\chi_{\rho}\chi_{m\re}+
	\chi_{\rho} Q_{\re}(s,z)\right)\mathcal{N}^{\re}_D(z)\gamma_0^{\re}\left(\chi_{2}\mathcal{R}_0(s)\chi_{\rho}\right),
\end{align*}
where we use that $\chi_2=1$ on $\partial\Omega^{\re}$. We recall that the free resolvent $s\mapsto \mathcal{R}_0(s)\in \mathcal{L}(L^2_{comp}(\mathbb{R}^3), H^1_{loc}(\mathbb{R}^3))$ is entire, see \cite[Theorem 3.1 and item 4 of the proof]{dyatlov_zworski}. The statement about the analyticity of $\chi_{\rho}\mathcal{D}^{\re}_D\chi_{\rho}$ then follows by combining this result together with Corollary \ref{cor:analyticity}.

Let us now obtain the bounds stated in Proposition \ref{prop:bounds}. We set $z=\re^{-1}$ (this choice comes from taking $z>0$ arbitrary, obtaining the desired bound, and then minimizing it in $z$). Let us define
\begin{align*}
	A^{\re}_0(s)&:=\mathcal{N}^{\re}_D(\re^{-1})\gamma_0^{\re}\mathcal{R}_0(s)\mathbb{1}_{\mathcal{O}},\qquad
	A^{\re}_1(s):=\left([\Delta, \chi_{2\varepsilon}]-(s^2-z^2)\chi_{2\varepsilon}\right)A^{\re}_0,\\
	B^{\re}_0(s)&:=\chi_{8\re}\mathcal{R}^{\re}_D(s)\chi_{4\re},\qquad
	B^{\re}_1(s):=[\Delta,\chi_{4\re}]B^{\re}_0,
\end{align*}
so that $
	\mathcal{D}^{\re}_D=-(\chi_{2\re}A^{\re}_0-\mathcal{R}_0^{\re}\chi_{8\re}B^{\re}_1A^{\re}_1-\chi_{4\re}B^{\re}_0A^{\re}_1),
$
where we took into account that $\chi_{8\re}=1$ on $\operatorname{supp}\chi_{4\re}$.

Combining Theorem \ref{prop:nz}, Proposition \ref{prop:nz_aux}, and using assumption that $0<\re<\re_0$ with $\re_0$ sufficiently small, yields:
\begin{align*}
	&\|A_0^{\re}\|_{L^2\rightarrow L^2}\lesssim \re^{1/2}\alpha_{\re}, \quad 	\|\nabla A_0^{\re}\|_{L^2\rightarrow L^2}\lesssim \re^{-1/2} \alpha_{\re}, \\
	&\alpha_{\re}=\|\mathbb{P}_0\gamma_0^{\re}\mathcal{R}_0(s)\|_{L^{2}(\mathcal{O}_0)\rightarrow L^2(\Gamma^{\re})}+\re^{1/2}\|\mathbb{P}_{\perp}\gamma_0^{\re}\mathcal{R}_0(s)\|_{L^{2}(\mathcal{O}_0)\rightarrow H^{1/2}(\Gamma^{\re})},\\
	&\alpha_{\re}\lesssim \re c_{\mathcal{O}_0}\mathrm{e}^{|\Re s|_{-}(R_0+\re)}(1+\re|s|), \quad  c_{\mathcal{O}_0}=R_0^{3/2}\max(1, r_0^{-2}).
\end{align*}
Next, 
\begin{align*}
	\|A_1^{\re}\|&\lesssim \re^{-1}\|\nabla A_0^{\re}\|+(|s\re|^2+1)\re^{-2}\|A_0^{\re}\|\lesssim  \re^{-3/2}(1+|s\re|^2)\alpha_{\re}.
\end{align*}
Using Corollary \ref{cor:analyticity}, 
\begin{align*}
\|B_0^{\re}(s)\|_{L^2\rightarrow H^{\ell}}&\lesssim \re^{2-\ell}(1+|s\re|)^{\ell-1}\mathrm{e}^{c_{\Omega}^+|\Re s|_{-}\re},\\
\|B_1^{\re}\|_{L^2\rightarrow L^2}&\lesssim \re^{-2}\|\chi_{8\re}\mathcal{R}^{\re}_D(s)\chi_{4\re}\|_{L^2\rightarrow L^2}+\re^{-1}\|\chi_{8\re}\mathcal{R}^{\re}_D(s)\chi_{4\re}\|_{L^2\rightarrow H^1}\lesssim \mathrm{e}^{c_{\Omega}^+|\Re s|_{-}\re}.
\end{align*}

Now that we have the preliminary bounds, we can prove the final bound on $\mathcal{D}^{\re}_D$. We start by fixing $R>r>1$ and studying the bound in the ball that excludes $\Omega^{\re}$:
\begin{align*}
	\|\chi_{r,R}\mathcal{D}^{\re}_D(s)\mathbb{1}_{\mathcal{O}}\|_{L^2\rightarrow L^2}\lesssim \|A_0^{\re}(s)\|+\|\chi_{r,R}\mathcal{R}_0^{\re}(s)\chi_{8\re}\|\|B_1^{\re}(s)\|\|A_1^{\re}(s)\|.
\end{align*}
Remark that $\chi_{r,R}\chi_{4\re}=0$ for $8\re<r$, and hence the operator $\chi_{4\re}B_0^{\re}A_1^{\re}$ does not occur in the above bound.

We have all the necessary bounds, and just need to employ Proposition \ref{prop:free_resolvent}. Altogether we obtain the sought far-field bound: 
\begin{align*}
	\|\chi_{r,R}\mathcal{D}_D^{\re}(s)\mathbb{1}_{\mathcal{O}}\|_{L^2\rightarrow L^2}&\lesssim 
\re^{1/2}\alpha_{\re}+R^{3/2} \frac{\mathrm{e}^{|\Re s|_{-}(R+(16+c_{\Omega}^+)\re)}}{r-16\varepsilon}\alpha_{\re}(1+|s\re|)^2\\
&\lesssim  \re R^{3/2}r^{-1}c_{\mathcal{O}_0}\mathrm{e}^{|\Re s|_{-}(R_0+R+17\re+c_{\Omega}^+\re)}(1+|s\re|)^3,
\end{align*}
where we used $R>r>1$. 
We set $\widetilde{c}_{\Omega}^+:=c_{\Omega}^++17$ and obtain the bound stated in the proposition.
\subsection{Proof of Theorem \ref{theorem:answer1}}
We will prove the following bound, which implies the bound stated in the theorem: 
	\begin{align}
		\label{eq:bdd}
	\|u^{\re}_{sc}(t)\|_{L^2(K_{\operatorname{ff}}^{\delta})}\leq C_{\operatorname{ff}}\re\min(1, \re^{k_{reg}-2}\mathrm{e}^{-\frac{\gamma_{\Omega}}{\re}(t-(R_{\operatorname{ff}}+R_0+\widetilde{c}_{\Omega}\re))})\sqrt{E_0^{k_{reg}}},
\end{align}
for some $\widetilde{c}_{\Omega}>0$. 

We fix $R>r>1$, and recall
 \eqref{eq:representation_formula_scat2} and \eqref{eq:ded}:
 \begin{align}
 	\label{eq:ueps2}
 	\chi_{r,R}u^{\varepsilon}_{sc}(t)=\frac{1}{2\pi i}\int_{i\mathbb{R}+\mu}\frac{e^{st}}{(\mu_*-s)^p}\chi_{r,R}\mathcal{D}^{\re}_D(s)\mathbb{1}_{\mathcal{O}}(su_0^p+v_0^p)ds, \quad 0<\mu<\mu_*.
 \end{align}
By Proposition \ref{prop:De}, we  can deform the contour to negative $\mu$. We choose it as a straight line $\Re s=-\gamma_{\Omega}/\re$, where $\gamma_{\Omega}>0$ is a positive constant s.t. $\gamma_{\Omega}<c_{\Omega}$ with $c_{\Omega}$ being like in Proposition \ref{prop:De}. This allows to rewrite 
\begin{align*}
\|\chi_{r,R}u^{\varepsilon}_{sc}(t)\|_{L^2}&\lesssim \int_{\mathbb{R}}\mathrm{e}^{-\gamma_{\Omega}t/\re}|\mu_*+\gamma_{\Omega}/\re+i\omega|^{-p}\|\chi_{r,R}\mathcal{D}^{\re}_D(\gamma_{\Omega}/\re+i\omega)\mathbb{1}_{\mathcal{O}}\|_{L^2\rightarrow L^2}\\
&\times (|\gamma_{\Omega}\re^{-1}+i\omega|\|u_0^p\|+\|u_1^p\|)d\omega,
\end{align*} 
and we use the bound of Proposition \ref{prop:De}. This leads to the following estimate: 
\begin{align*}
\|\chi_{r,R}u^{\re}_{sc}(t)\|_{L^2}\lesssim \re \mathrm{e}^{-\frac{\gamma_{\Omega}}{\re}(t-R-R_0-\tilde{c}_{\Omega}\re)}C_{r,R}c_{\mathcal{O}_0}S^{\re}_p(\|u_0^p\|+\|u_1^p\|),	
\end{align*}
where the integral below converges if $p>5$:
\begin{align*}
	S^{\re}_p&=\int_{\mathbb{R}}|\mu_*+\gamma_{\Omega}\re^{-1}+i\omega|^{-p}\max(|\gamma_{\Omega}\re^{-1}+i\omega|,1)|1+\gamma_{\Omega}+i\omega \re|^3 d\omega\lesssim C_{p,\gamma_{\Omega}}\re^{p-2},
\end{align*}
 The assumption of the theorem and Remark \ref{rem:bmap} enable the choice $p=k_{reg}$. We thus have obtained one part of \eqref{eq:bdd}.

The above bound is clearly convenient to use when $t>R+R_0+\widetilde{c}_{\Omega}\re$, however, otherwise, for a fixed $t$ that does not satisfy this condition, it is non-optimal when $R\rightarrow +\infty$. Thus, in this regime we use the contour with $\Re s=0$:
\begin{align*}
\|\chi_{r,R}&u^{\re}_{sc}(t)\|_{L^2}\lesssim \re C_{r,R}c_{\mathcal{O}_0}\max(1, r_0^{-1})(\|u_0^p\|+\|u_1^p\|)\\
& \times\int_{\mathbb{R}}|\mu_*+i\omega|^{-p}(1+|\omega| \re)^3\max(1, |\omega|)d\omega\lesssim C_p\re C_{r,R}c_{\mathcal{O}_0}(\|u_0^p\|+\|u_1^p\|).
\end{align*}

\section{Long-time convergence of the asymptotic model of \cite{sini_wang_yao}}
\label{sec:asymptotic}
In \cite{sini_wang_yao}, it was suggested to approximate the scattered field by the following expression, considered also in \cite{martin}: 
\begin{align}
	\label{eq:as1}
	u^{\re}_{app}(t,\bx):=-c^{\re}\frac{u^{inc}(t-\|\bx\|, \mathbf{0})}{4\pi\|\bx\|}, 
\end{align}
where $c^{\re}$ is the capacitance of $\Gamma^{\re}$, defined as follows. We start by introducing
\begin{align*}
	\sigma^{\re}\in L^2(\Gamma^{\re}) \text{ as a unique $H^{-1/2}(\Gamma^{\re})$-solution to }\int_{\Gamma^{\re}}\frac{\sigma^{\re}(\by)}{4\pi\|\bx-\by\|}d\Gamma_{\by}=1, \quad \bx\in \Gamma^{\re}.
\end{align*}
The capacitance is then   $c^{\re}:=\int_{\Gamma^{\re}}\sigma^{\re}(\by)d\Gamma_{\by}=\re c^1$.

 As shown in the statement below, it indeed provides an approximation of the scattered field \textit{in the far-field}, which does not deteriorate in accuracy with a long time. 
\begin{theorem}
	\label{theorem:answer2}
 There exists $\re_0>0$, s.t. for all $0<\re<\re_0$, the error between the exact scattered field $u^{\re}_{sc}$ and its  approximation $u^{\re}_{app}$, defined in \eqref{eq:as1}, behaves as follows. For $k_{reg}\geq 7$,   $t=R_0+R_{\operatorname{ff}}+\tau$, 
\begin{align*}
	\|u^{\re}_{sc}(t)-u^{\re}_{app}(t)\|_{L^{2}(K_{\operatorname{ff}})}\lesssim 		 \left\{
	\begin{array}{ll}
		\widetilde{C}_{\operatorname{ff}}\re^2   \sqrt{E_0^{k_{reg}}},& \tau\leq 0,	 		\\	 C_{\operatorname{ff}}\re^{k_{reg}-1}\mathrm{e}^{-\frac{\gamma_{\Omega}\tau}{\re}}\sqrt{E_0^{k_{reg}}}, &\tau >0,
	\end{array}
	\right.
\end{align*}
where  $\gamma_{\Omega}$, $C_{\operatorname{ff}}$ are like in Theorem \ref{theorem:answer1}, 
and the constant $\widetilde{C}_{\operatorname{ff}}$ depends on $R$, $r$, $R_0$, $r_0$ and the obstacle $\Omega$. 
\end{theorem}
\begin{proof}
	The proof is decomposed into two parts. 
	
	\textbf{Proof for $t>R_0+R_{\operatorname{ff}}$. }
	By the strong Huygens principle and finite speed of wave propagation, since the initial data $u_0, \, v_0$ are supported inside $B(0, R_0)\setminus \overline{B(0, r_0)}$, we have
	\begin{align*}
	\operatorname{supp}u^{inc}(t)\subseteq \cup_{\vec{x}\in \overline{\mathcal{O}_0}}\{\vec{y}\in \mathbb{R}^3: \|\vec{x}-\vec{y}\|=t\}\subseteq \{\vec{y}\in \mathbb{R}^3: \, t-R_0\leq \|\vec{y}\|\leq t+R_0\}.
	\end{align*}
We conclude that 
$u^{inc}(t-\|\bx\|,\vec{0})=0$  for $t>\|\bx\|+R_0 \text{ and }t<\|\bx\|-R_0$, and hence for $t>R_0+R_{\operatorname{ff}}$, $
\|u^{\re}_{app}(t)\|_{L^2(K_{\operatorname{ff}})}=0.$ 
This shows that for $t>R_{\operatorname{ff}}+R_0$, $
	\|u^{\re}_{sc}(t)-u^{\re}_{app}(t)\|_{L^2(K_{\operatorname{ff}})}= \|u^{\re}_{sc}(t)\|_{L^2(K_{\operatorname{ff}})},$
and it suffices to apply the result of Theorem \ref{theorem:answer1}. 
	
\textbf{Proof for $t\leq R_0+R_{\operatorname{ff}}$.} See Appendix \ref{appendix:error}.	
\end{proof}
This result shows that the long-time error of the asymptotic model behaves like the scattered field, which itself decays rapidly as $t\rightarrow +\infty$, see the discussion after Theorem \ref{theorem:answer1}. For times where we expect the scattered field to be $O(\re)$ (i.e. where using the asymptotic model is of interest), the model produces the absolute error of $O(\re^2)$.  

\textbf{One answer to the question posed in the title of the paper.} Comparing the estimates on the field of Theorem \ref{theorem:answer1} and error estimates of Theorem  \ref{theorem:answer2}, we conclude that obstacles can be considered 'small', in other words, well-approximated by the point scatterer \eqref{eq:as1}, when, with some $\kappa>0$,  
$
	\re \sqrt{E_0^7}\ll \re^{\kappa},
$
which ensures the convergence of the approximated scattered field of $O(\re^{1+\kappa})$. The requirement of the regularity of the initial data is hardly optimal, and stems from numerous estimates non-optimal in terms of the frequency parameter $s$ we have made before in the paper.

\section{Discussion and extension of the results}
We have studied the problem of the sound-soft scattering by a small 3D star-shaped particle of an incident wave that obeys the wave equation with inhomogeneous initial conditions, and have proven that replacing the particle by a point source does not deteriorate the solution error in time.
An interesting extension is the 2D case, where the Huygens principle no longer holds true; we remark that the techniques used in this paper apply almost verbatim, with the difficulty being passing from the Laplace domain to the time domain, see  \cite{christiansen_datchev_yang}. Another possible extension of the work is the analysis of the scattering by many particles, or of the situations of resonances. 
\section*{Acknowledgments}
The first author is grateful to Andrea Mantile (University of Reims, France) for many fruitful discussions, and to Mourad Sini (RICAM, Austria) for asking the question that is in the origin of this work.
\appendix

\section{Proof of the second part of Theorem \ref{theorem:answer2}}
\label{appendix:error}
We proceed as follows. First of all, in Section \ref{sec:td_rep}, we define a convenient representation of the approximated scattered field via a Laplace contour integral, analogous to \eqref{eq:ueps}. The error analysis then will reduce to the frequency-domain error analysis of the integrand, see Section \ref{sec:lpl_dom}. Combining the Laplace domain representation with the frequency domain analysis is a subject of Section \ref{sec:fin_proof}.
\subsection{Time-domain representation formula}
\label{sec:td_rep}
First of all, let us define an approximated frequency-domain single-layer potential:
\begin{align*}
		(\mathcal{S}^{\re}_{app}\varphi)(\bx):=\frac{\mathrm{e}^{-s\|\bx\|}}{4\pi\|\bx\|}\int_{\Gamma^{\re}}\varphi(\bx)d\Gamma_{\bx}, \quad s\in \mathbb{C}^+, \quad \varphi\in L^2(\Gamma^{\re}).
\end{align*}
For sufficiently regular $q$, we define an operator of evaluation in $\vec{0}$: 
$P_{\vec{0}}q=q(\vec{0}).$ 
The following result is evident from the explicit formula for $\hat{u}^{\re}$. 
\begin{proposition}
	\label{prop:holomorph}
	For all $\rho>0$, the function $\mathbb{C}\ni s\mapsto \chi_{\rho}\hat{u}^{\re}_{app}(s,.)$ is $L^2(\Omega^{\re,c})$-entire.
\end{proposition}
Our true starting point is the following representation formula, which will allow to compare the scattered field given by  \eqref{eq:representation_formula_main} and its approximation. 
\begin{proposition}
The approximate field can be written in the following form:
	\begin{align}
	\label{eq:represenation_formula_uapp}
	{u}^{\re}_{app}(t)=-\frac{1}{2\pi i}\int_{\Re s=\mu}	\frac{\mathrm{e}^{st}}{(\mu_*-s)^p}\mathcal{S}^{\re}_{app}(s)\sigma^{\re}P_{\vec{0}}(\mathcal{R}_0(s)(s{u}_0^p+{v}_0^p)) ds, \quad 
	\end{align}
$p\geq 2, \, 0<\mu<\mu_*$, 
provided that $k_{reg}\geq p$, 
with the integral understood as a Bochner integral in $L^2(\Omega^{\re,c}\cap B_R(0)),$ $R>0, $ for all fixed $t\geq 0$. 
\end{proposition}
\begin{proof}
Remark that the function  $
	t\mapsto (\bx\mapsto\frac{u^{inc}(t-\|\bx\|, \vec{0})}{4\pi\|\bx\|}\in L^2(\Omega^{\re,c}))$ 
is causal, since $u^{inc}$ is such. 
Then applying the Laplace transform to  $u^{\re}_{app}=-\frac{u^{inc}(t-\|\bx\|, \vec{0})}{4\pi\|\bx\|}$ and using the definitions of $\mathcal{S}^{\re}$, $c^{\re}$, $P_{\mathbf{0}}$ yields the identity:
\begin{align}
	\label{eq:two_eq}
	\hat{u}^{\re}_{app}=-\frac{\mathrm{e}^{-s\|\bx\|}}{4\pi\|\bx\|}\hat{u}^{inc}(s,\vec{0})c^{\re}=-\mathcal{S}^{\re}_{app}(s)\sigma^{\re}P_{\mathbf{0}}\hat{u}^{inc}(s).
\end{align}
By Lemma \ref{lem:resolvent_B}, $
	\hat{u}^{inc}(s)=\mathcal{R}_0(s)(su_0+u_1)=\vec{e}_{x}\cdot (s-\mathbb{B})^{-1}\vec{u}_0.$ 
We fix $\mu_*$ like in \eqref{eq:representation_formula}, take $p\in \mathbb{N}: \, p\geq 2$, and rewrite 
\begin{align}
	\label{eq:uincsp}
	\hat{u}^{inc}(s)&=\vec{e}_{x}\cdot (s-\mathbb{B})^{-1}(\mu_*-\mathbb{B})^{-p}(\mu_*-\mathbb{B})^p\vec{u}_0=\vec{e}_{x}\cdot (\mu_*-\mathbb{B})^{-p}(s-\mathbb{B})^{-1}\vec{u}_0^p.
\end{align}
Let us now consider the following difference, cf. with the above: 
\begin{align}
	\label{eq:defdps}
\boldsymbol{d}_p(s):=\left((\mu_*-\mathbb{B})^{-p}-(\mu_*-s)^{-p}\right) (s-\mathbb{B})^{-1}\vec{u}_0^p.
\end{align}
Remark that 
\begin{align*}
(\mu_*-\mathbb{B})^{-p}-(\mu_*-s)^{-p}&=(\mu_*-s)^{-p}(\mu_*-\mathbb{B})^{-p}((\mu_*-s)^p-(\mu_*-\mathbb{B})^{p})\\
&=(\mu_*-s)^{-p}(\mu_*-\mathbb{B})^{-p}(\mathbb{B}-s)\mathcal{P}(s,\mu_*,\mathbb{B}),
\end{align*}
where $\mathcal{P}(s,\mu_*,\mathbb{B})$ is a polynomial of order $p-1$ in $\mathbb{B}$, $s$, $\mu_*$. Therefore, 
\begin{align}
	\label{eq:bdp}
	\boldsymbol{d}_p(s)=-(\mu_*-s)^{-p}\mathcal{P}(s,\mu_*,\mathbb{B})(\mu_*-\mathbb{B})^{-p}\vec{u}_0^p=-(\mu_*-s)^{-p}\mathcal{P}(s,\mu_*,\mathbb{B})\vec{u}_0.
\end{align}
By \eqref{eq:defdps} and \eqref{eq:uincsp},
\begin{align*}
	\hat{u}^{inc}(s)&=(\mu_*-s)^{-p}\vec{e}_x\cdot(s-\mathbb{B})^{-1}\vec{u}_0^p+\vec{e}_x\cdot\vec{d}_p(s)\\
	&=(\mu_*-s)^{-p}\mathcal{R}_0(s)(su_0^p+v_0^p)+\vec{e}_x\cdot\vec{d}_p(s),
\end{align*}
see Lemma \ref{lem:resolvent_B} for the last expression. We plug in the above expression into \eqref{eq:two_eq}. We remark first of all that ${P}_{\vec{0}}(\vec{e}_x\cdot \vec{d}_p(s))=0$. This follows from \eqref{eq:bdp}, by recalling that $\mathcal{P}(s,\mu_*, \mathbb{B})$ is a differential operator, thus $\operatorname{supp}\mathcal{P}(s,\mu_*, \mathbb{B})\vec{u}_0\subseteq \operatorname{supp}\vec{u}_0$. Therefore, 
\begin{align*}
	\hat{u}^{\re}_{app}(s)=-(\mu_*-s)^{-p}\mathcal{S}^{\re}_{app}(s)\sigma^{\re}P_{\vec{0}}\left(\mathcal{R}_0(s)(su_0^p+v_0^p)\right).
\end{align*}
Remark that $
	|(\mathcal{S}^{\re}_{app}(s)\sigma^{\re})(\bx)|\leq c^{\re}\frac{\mathrm{e}^{-\Re s\|\bx\|}}{4\pi\|\bx\|}. $
Since, additionally, $p\geq 2$, and $\vec{u}_0^p$ is in $L^2(\mathbb{R}^3)$ (guaranteed by the regularity assumption of the proposition), the function $
	s\mapsto \mathrm{e}^{st}(\mu_*-s)^{-p}P_{\vec{0}}\left(\mathcal{R}_0(s)(su_0^p+v_0^p)\right)$ 
is integrable on a line $\Re s=\mu$, $\mu<\mu_*$, due to the straightforward free resolvent bounds of Proposition \ref{prop:free_resolvent}. 

Therefore, the integral in the lhs is a Bochner $L^2(\Omega^{\re,c}\cap B_R(0))$-integral:
\begin{align*}
	 -\int_{\Re s=\mu}(\mu_*-s)^{-p}\mathrm{e}^{st}\mathcal{S}^{\re}_{app}(s)\sigma^{\re}P_{\vec{0}}\left(\mathcal{R}_0(s)(su_0^p+v_0^p)\right)ds=\int_{\Re s=\mu}\mathrm{e}^{st}\hat{u}^{\re}_{app}(s)ds.
\end{align*}
The rhs is equal to $2\pi i u^{\re}_{app}(t)$ by Theorem 4.2.21 in \cite{vv_lpl}.\qed

\end{proof}
\subsection{Laplace-domain error bounds}
\label{sec:lpl_dom}
Recall that the scattering field can be represented with the help of the operator  $\mathcal{D}^{\re}_D(s)$, as defined in \eqref{eq:ded}, acting on the  initial data. We introduce an approximation of $\mathcal{D}^{\re}_D(s)$, $\Re s>0$, 
\begin{align}
	\label{eq:dapp}
	\mathcal{D}_{app}^{\re}(s)=-c^{\re}\frac{\mathrm{e}^{-s\|\bx\|}}{4\pi\|\bx\|}P_{\mathbf{0}}\mathcal{R}_0(s)\equiv -\mathcal{S}^{\re}_{app}(s)\sigma^{\re}P_{\vec{0}}\mathcal{R}_0(s) \in \mathcal{L}(L^2(\mathcal{O}), L^2(\Omega^{\re,c})). 
\end{align}
Indeed, $\hat{u}^{\re}_{app}=\mathcal{D}^{\re}_{app}(su_0+u_1)$. 

Let us state the representation formula, which combines \eqref{eq:representation_formula_main} and \eqref{eq:represenation_formula_uapp}, and definitions \eqref{eq:dapp} and \eqref{eq:ded}:
\begin{align}
	\label{eq:ere}
	e^{\re}(t):=	u^{\re}_{app}(t)-u^{\re}_{sc}(t)&=\frac{1}{2\pi i}\int_{\Re s=\mu}\frac{\mathrm{e}^{st}}{(\mu_*-s)^p}(\mathcal{D}^{\re}_{app}(s)-\mathcal{D}^{\re}_D(s))(su_0^p+v_0^p)ds,
\end{align}
where $p\geq 2, \, 0<\mu<\mu_*.$

We see that to evaluate the error between the scattered and the approximate fields, we need to quantify  $\mathcal{D}_{app}^{\re}(s)-\mathcal{D}^{\re}_D(s)$; this will be done separately for 'low' and 'high' frequencies. In the context of asymptotic time-domain modelling, the high-frequency error bound can be obtained based on stability estimates, and, in principle, for the data smooth in time can be made arbitrarily small, cf. Proposition \ref{prop:hf}. The low frequency errors are the dominant ones. This has been observed also in \cite{korikov} and \cite{korikov_plamenevskii}. We start with the high-frequency estimates. 
\begin{remark}
We repeatedly use the constants defined in Proposition \ref{prop:bounds}: 
\begin{align*}
	C_{r,R}=R^{3/2}r^{-1}, \quad c_{\mathcal{O}_0}=R_0^{3/2}\max(1,r_0^{-2}).
\end{align*}
For operators ${A}: \, L^2(\mathbb{R}^3)\rightarrow L^2(\Omega^{\re,c})$, we use the obvious notation   $\|A\|_{L^2(K_0)\rightarrow L^2(K)}:=\|\mathbb{1}_KA\mathbb{1}_{K_0}\|_{L^2(\mathbb{R}^3)\rightarrow L^2(\Omega^{\re,c})}$. 
\end{remark}
\begin{proposition}[High-frequency estimate]
	\label{prop:hf}
Let $\ell\geq 0$, $R>r>1$. There exists $\re_0>0$, s.t. for all $\theta>0$, $0<\re<\re_0$, and for all $s\in \mathscr{D}^{\re}$ with $|s\re|>\theta$, 
	\begin{align*}
		\|\mathcal{D}_{app}^{\re}(s)-\mathcal{D}^{\re}_D(s)\|_{L^2(\mathcal{O}_0)\rightarrow L^2(B_{r,R})}&\lesssim \re^{\ell+1} |s|^{\ell}\theta^{-\ell}\\
		&\times C_{r,R}c_{\mathcal{O}_0}
	 \mathrm{e}^{|\Re s|_{-}(R_0+R)}(1+|s\re|)^3,
\end{align*}
where the hidden constant depends on $\Omega$, $\re_0$ and $\theta$ only. 
\end{proposition}
Before proving this result, let us state an auxiliary bound.
\begin{lemma}
	\label{lem:estsigma}
	Let $R>r>1$. Then, there exist $\re_0, \, C>0$, s.t. for all $0<\re<\re_0$, all $s\in\mathbb{C}^+$, \begin{align*}
	\|\sigma^{\re}P_{\vec{0}}\mathcal{R}_0(s)\|_{L^2(\mathcal{O}_0)\rightarrow L^2(\Gamma^{\re})}&\leq C \mathrm{e}^{|\Re s|_{-}(R_0+\re)}R_0^{3/2}r_0^{-1},\\
	\|\mathcal{D}^{\re}_{app}(s)\|_{L^2(\mathcal{O}_0)\rightarrow L^2(B_{r,R})}&\leq C \re C_{r,R}\mathrm{e}^{|\Re s|_{-}(R_0+R+2\re)}R_0^{3/2}r_0^{-1}.
	\end{align*}
\end{lemma}
\begin{proof}
	One verifies that $\sigma^{\re}(\bx)$ defined after \eqref{eq:as1} satisfies $\sigma^{\re}(\bx)=\re^{-1}\sigma^1(\re^{-1}\bx)$, therefore $\|\sigma^{\re}\|_{L^2(\Gamma^{\re})}\lesssim 1$. The sought first bound follows from  $\|P_{\vec{0}}\mathcal{R}_0(s)\|_{L^2(\mathcal{O}_0)\rightarrow \mathbb{C}}\lesssim \|\mathcal{R}_0(s)\|_{L^2(\mathcal{O}_0)\rightarrow L^{\infty}(B_{\re})}$ and using Lemma \ref{lem:rz}.
	
	It remains to bound 
	\begin{align*}
		\|\mathcal{D}_{app}^{\re}(s)\|\leq \|\mathcal{S}^{\re}_{app}(s)\|_{L^2(\Gamma^{\re})\rightarrow L^2(B_{r,R})}\|\sigma^{\re}P_{\vec{0}}\mathcal{R}_0(s)\|_{L^2(\mathcal{O}_0)\rightarrow L^2(\Gamma^{\re})}.
	\end{align*}
	We then have the following, for each $g\in L^2(\Gamma^{\re})$: $$
	\|\mathcal{S}^{\re}_{app}(s)g\|^2_{L^2(B_{r,R})}\lesssim R^{3}\|\mathcal{S}^{\re}_{app}(s)g\|^2_{L^{\infty}(B_{r,R})}\lesssim \int_{\Gamma^{\re}}\frac{\mathrm{e}^{-2|\Re s|_{-}(R+\re)}}{4\pi\|r-\re\|^2}d\Gamma_{\by}\|g\|^2_{L^2(\Gamma^{\re})},$$
	the latter inequality following by the Cauchy-Schwarz inequality. The desired result is straightforward by combining the above and the previously obtained bound.
\end{proof}
\begin{proof}[Proof of Proposition \ref{prop:hf}]
Omitting the index in the norm  $\|.\|_{L^2(\mathcal{O})\rightarrow L^2(B_{r,R})}\equiv \|.\|$, we start by using a triangle inequality combined with $|s\re|\theta^{-1}>1$ and $\ell\geq 0$:
\begin{align*}
	\|\mathcal{D}^{\re}_{app}(s)-\mathcal{D}^{\re}(s)\|\leq |s\re|^{\ell}\theta^{-\ell}(	\|\mathcal{D}^{\re}_{app}(s)\|+\|\mathcal{D}^{\re}_D(s)\|).
\end{align*}	
We then combine the bounds of Proposition \ref{prop:bounds} and Lemma \ref{lem:estsigma},  and use that $s\in\mathscr{D}^{\re}$, thus $|\Re s|_{-}\re \leq c_{\Omega}$.  
\qed
\end{proof}
The low-frequency analysis is slightly more involved. 
\begin{proposition}
	\label{prop:lf}
	Let $R>r>1.$ There exist $\theta>0$ and $\re_0>0$ sufficiently small, s.t. for all $0<\re<\re_0$, all $s\in \mathscr{D}^{\re}$, s.t. $|s\re|\leq \theta$,  
	\begin{align*}
		\|\mathcal{D}^{\re}_{app}(s)-\mathcal{D}^{\re}_{D}(s)\|_{L^2(\mathcal{O}_0)\rightarrow L^{2}(B_{r,R})}&\lesssim \re^2(1+|s|+r^{-1})c_{r,R}c_{\mathcal{O}_0}\mathrm{e}^{|\Re s|_{-}(R_0+R)},
	\end{align*}
with the hidden constant depending on $\Omega$, $\theta$ and $\re_0$ only. 
\end{proposition}
To prove this result, we will make use of a representation formula for the exact solution $\hat{u}^{\re}_{sc}$ adapted to our needs. Recalling the definition of the single-layer potential $\mathcal{S}^{\re}(s)\in \mathcal{L}(H^{-1/2}(\Gamma^{\re}), H^1(\Omega^{\re,c}))$  
\begin{align*}
	\mathcal{S}^{\re}(s)\varphi=\int_{\Gamma^{\re}}\frac{\mathrm{e}^{-s\|\bx-\by\|}}{4\pi\|\bx-\by\|}\varphi(\by)d\Gamma_{\by}, \quad \bx\in \Omega^{\re,c},
\end{align*}
and its trace  $\mathbb{S}^{\re}(s)=\gamma_0^{\re}\mathcal{S}^{\re}(s)\in \mathcal{L}(H^{-1/2+\kappa}(\Gamma^{\re}), H^{1/2+\kappa}(\Gamma^{\re}))$, $\kappa\geq 0$, let us recall that for $s\in \mathbb{C}^+$, it holds that 
\begin{align}
	\label{eq:sred}
	\mathcal{D}^{\re}_D(s)=\mathcal{R}_D^{\re}(s)-\mathcal{R}_0^{\re}(s)=-\mathcal{S}^{\re}(s)(\mathbb{S}^{\re}(s))^{-1}\gamma_0^{\re}\mathcal{R}_0(s), \quad s\in \mathbb{C}^+, 
\end{align}
cf. \eqref{eq:hatug} and theory of boundary integral equations, cf. e.g.  \cite{sauter_schwab} or \cite[Chapter 9]{taylor_pdes}. 
Remark that $\mathbb{S}^{\re}(s)\in \mathcal{L}(H^{-1/2}(\Gamma^{\re}), H^{1/2}(\Gamma^{\re}))$ is indeed invertible for all $s\in \mathbb{C}^+$, and actually everywhere in $\mathbb{C}$ apart from a discrete subset \cite[Proposition 7.10 in Charpter 9]{taylor_pdes}. The mapping $s\mapsto\chi_{\rho}\mathcal{S}^{\re}(s)\in \mathcal{L}(H^{-1/2}(\Gamma^{\re}), H^1(\Omega^{\re,c}))$ can be shown to be an analytic function of $s\in \mathbb{C}$. 

Since the mapping $s\mapsto \chi_{\rho}	(\mathcal{R}_0(s)-\mathcal{R}_D^{\re}(s))\chi_{\rho}\in \mathcal{L}(L^2(\Omega^{\re,c}))$ is meromorphic on $\mathscr{D}^{\re}$, the same holds true for $\chi_{\rho}\mathcal{S}^{\re}(s)(\mathbb{S}^{\re}(s))^{-1}\gamma_0^{\re}\mathcal{R}_0(s)\chi_{\rho}\in \mathcal{L}(L^2(\Omega^{\re,c}))$, and their poles coincide. Thus $s\mapsto-\chi_{\rho}\mathcal{S}^{\re}(s)(\mathbb{S}^{\re}(s))^{-1}\gamma_0^{\re}\mathcal{R}_0(s)\chi_{\rho}$ admits a unique analytic continuation from $\mathbb{C}^+$ into $\mathscr{D}^+$.

Using \eqref{eq:sred} and \eqref{eq:dapp}, we split the error of Proposition \ref{prop:lf} into two parts: 
\begin{align}
		\label{eq:uscre}
	\mathcal{E}^{\re}:=\mathcal{D}^{\re}_{app}-\mathcal{D}^{\re}_{D}=-(\mathcal{S}^{\re}_{app}-\mathcal{S}^{\re})\sigma^{\re}P_{\vec{0}}\mathcal{R}_0+\mathcal{S}^{\re}\left((\mathbb{S}^{\re})^{-1}\gamma_0^{\re}\mathcal{R}_0-\sigma^{\re}P_{\vec{0}}\mathcal{R}_0\right),
\end{align}
and estimate the above two quantities separately.  
\begin{lemma}
	\label{lem:err_lmbd}
There exist $\theta>0$ and $\re_0>0$, s.t. for all $0<\re<\re_0$, $s\in \mathscr{D}^{\re}$ with 
	$|s\re|\leq \theta$, it holds that 
		\begin{align*}
		\|(\mathbb{S}^{\re}(s))^{-1}\gamma_0^{\re}\mathcal{R}_0(s)-\sigma^{\re}P_{\vec{0}}\mathcal{R}_0(s)\|_{L^2(\mathcal{O}_0)\rightarrow L^2(\Gamma^{\re})}&\lesssim \re (1+|s|)c_{\mathcal{O}_0}\mathrm{e}^{|\Re s|_{-}R_0}
	\end{align*}
where the hidden constant depends on $\Omega$, $\theta$, $\re_0$. 
\end{lemma}
\begin{proof}
Let $f\in L^2(\mathcal{O}_0)$. Let us define two boundary densities: 
\begin{align}
\hat{\lambda}^{\re}:=-(\mathbb{S}^{\re})^{-1}\gamma_0^{\re}v\; \text{ and }\; \hat{\lambda}^{\re}_{app}:=-\sigma^{\re}P_{\vec{0}}v, \quad v:=\mathcal{R}_0(s)f.
\end{align}
We start by proving the bound (recall that by elliptic regularity $v\in C^{\infty}(\overline{B_1})$): 
\begin{align}
	\label{eq:first_estimate}
	\|\hat{\lambda}^{\re}-\hat{\lambda}^{\re}_{app}\|_{L^2(\Gamma^{\re})}&\lesssim \re\|\nabla v\|_{L^{\infty}(B_{\re})}+\re \|sv\|_{L^{\infty}(B_{\re})}.
\end{align}
For this we will rewrite $\hat{\lambda}^{\re}$ in a form convenient for comparison with $\hat{\lambda}^{\re}_{app}$. 

Let us define the operator $\mathcal{J}^{\re}: \, L^2(\Gamma^{\re})\rightarrow L^2(\Gamma)$ by $(\mathcal{J}^{\re}\varphi)(\vec{x}):=\varphi(\re\vec{x})$, $\vec{x}\in \Gamma$. 
Remark that $\mathcal{J}^{\re}$ is obviously invertible and satisfies 
\begin{align}
	\label{eq:jre}
	\|\mathcal{J}^{\re}\varphi\|_{L^2(\Gamma)}=\re^{-1}\|\varphi\|_{L^2(\Gamma^{\re})}.
\end{align}
Recalling that 
$
	(\mathbb{S}^{\re}(s)\varphi)(\bx)=\int_{\Gamma^{\re}}\frac{\mathrm{e}^{-s\|\bx-\by\|}}{4\pi\|\bx-\by\|}\varphi(\by)d\Gamma_{\by}, \quad \bx\in \Gamma^{\re},
$
we rewrite, using the scaling argument,
\begin{align*}
\mathbb{S}^{\re}(s)\hat{\lambda}^{\re}=-\gamma_0^{\re}v\iff\re \mathbb{S}^{1}(\re s)\Lambda^{\re}=-\mathcal{J}^{\re}(	\gamma_0^{\re}v), \quad \Lambda^{\re}=\mathcal{J}^{\re}\hat{\lambda}^{\re}.
\end{align*}
The above further rewrites, with $\mathbb{S}(s):=\mathbb{S}^1(s)$ and $\mathbb{S}_0:=\mathbb{S}(0)$:
\begin{align}
\label{eq:csrho}
\re\left(\mathbb{S}_0\hat{\Lambda}^{\re}+\mathbb{E}(s\re)\hat{\Lambda}^{\re}\right)=-\mathcal{J}^{\re}\gamma_0^{\re}v,\quad \mathbb{E}(s)=\mathbb{S}(s)-\mathbb{S}(0)=\mathbb{S}(s)-\mathbb{S}_0.
\end{align}
By the standard result about the norm of the Hilbert-Schmidt operator \cite[p.267]{conway}
\begin{align*}
	\|\mathbb{E}(s\re)\|_{L^2(\Gamma)\rightarrow L^2(\Gamma)}^2\leq \iint_{\Gamma\times \Gamma}\left|\frac{\mathrm{e}^{-s\re\|\bx-\by\|}-1}{4\pi\|\bx-\by\|}\right|^2 d\Gamma_{\bx}d\Gamma_{\by}\leq C_{\Gamma}\mathrm{e}^{4\theta}|s\re|^2,
\end{align*}
where in the last inequality we employed the mean-value theorem  $ \left|\mathrm{e}^{-s\re\|\bx-\by\|}-1\right|\leq |s\re|\mathrm{e}^{|s\re|\operatorname{diam}\Gamma}\|\bx-\by\|$ and $|s\re|\operatorname{diam}\Gamma\leq 2\theta$. For $\varphi\in L^2(\Gamma)$,  
\begin{align*}
(\nabla \mathbb{E}(s\re)\varphi)(\vec{x})&=\int_{\Gamma }\frac{(\mathrm{e}^{-s\re\|\bx-\by\|}-1)(\vec{y}-\vec{x})}{4\pi\|\bx-\by\|^3}\varphi(\by)d\Gamma_{\by}\\
&+s\re\int_{\Gamma}\frac{\mathrm{e}^{-s\re\|\bx-\by\|}(\vec{y}-\vec{x})}{4\pi\|\bx-\by\|^2}\varphi(\by)d\Gamma_{\by}.
\end{align*}
Again using the mean-value theorem to estimate the first integral, 
\begin{align*}
	\|\nabla \mathbb{E}(s\re)\varphi\|_{L^2(\Gamma)}^2\lesssim
\int_{\Gamma}\left(\int_{\Gamma}
	 \frac{\mathrm{e}^{\theta \operatorname{diam}\Gamma} |s\re|}{4\pi\|\bx-\by\|}|\varphi(\by)|d\Gamma_{\by}\right)^2d\Gamma_{\bx}\lesssim \mathrm{e}^{4\theta}|s\re|^2\|\mathbb{S}_0 |\varphi|\,\|_{L^2(\Gamma)}^2.
\end{align*}
By \cite{verchota}, $\mathbb{S}_0: \, L^2(\Gamma)\rightarrow H^1(\Gamma)$ is bounded and invertible. Therefore, 
\begin{align}
	\label{eq:ab_bds}
(a)\quad\|\mathbb{E}(s\re)\|_{L^2\rightarrow H^1}\lesssim \mathrm{e}^{2\theta}|s\re| \quad \text{ and }\quad (b)\quad \|\mathbb{S}^{-1}_0\mathbb{E}(s\re)\|_{L^2\rightarrow L^2}\lesssim \mathrm{e}^{2\theta}|s\re|,
\end{align}
with hidden constants independent of $\theta$, $s$, $\re$. 
 By the Neumann series argument applied to \eqref{eq:csrho}, and \eqref{eq:ab_bds}(b), choosing $\theta$ sufficiently small (recall $|s\re|<\theta$) yields the following expansion:
\begin{align*}
	&\hat{\Lambda}^{\re}=-\re^{-1}\mathbb{S}^{-1}_0\mathcal{J}^{\re}\gamma_0^{\re}v+\re^{-1}R^{\re},\text{ with }\\
	&\|R^{\re}\|_{L^2(\Gamma)}\lesssim \|\mathbb{S}_0^{-1}\mathbb{E}(s\re)\|_{L^2\rightarrow L^2}\|\mathcal{J}^{\re}\gamma_0^{\re}v\|_{L^2(\Gamma)}\overset{\eqref{eq:jre}}{\lesssim} |s|\|\gamma_0^{\re}v\|_{L^2(\Gamma^{\re})}\lesssim |s\re|\|\gamma_0^{\re}v\|_{L^{\infty}(\Gamma^{\re})}.
	\end{align*}
To compare the above expression to $\hat{\lambda}^{\re}_{app}=-\sigma^{\re}v(s,\vec{0})$, we rewrite further
\begin{align*}
\hat{\Lambda}^{\re}(s)&=-\re^{-1}\mathbb{S}^{-1}_0v(s,\vec{0})+\re^{-1}\mathbb{S}_0^{-1}\left(v(s,\vec{0})-\mathcal{J}^{\re}\gamma_0^{\re}v(s)\right)+\re^{-1}R^{\re}(s).
\end{align*}
Recall that $\mathbb{S}^{-1}_0 1=\sigma^1$ (cf. the identity after \eqref{eq:as1}), and, with the scaling argument, $\mathcal{J}^{\re}\hat{\lambda}^{\re}_{app}=-v(s,\vec{0})\re^{-1}\sigma^{1}$. Therefore, by \eqref{eq:jre},  
\begin{align*}
	\|\hat{\lambda}^{\re}-\hat{\lambda}^{\re}_{app}\|_{L^2(\Gamma^{\re})}&\leq  \|\mathbb{S}_0^{-1}\left(v(s,\vec{0})-\mathcal{J}^{\re}\gamma_0^{\re}v\right)\|_{L^2(\Gamma)}+\|R^{\re}\|_{L^2(\Gamma)}\\
	&\lesssim  \|v(s,\vec{0})-\mathcal{J}^{\re}\gamma_0^{\re}v\|_{L^2(\Gamma)}+\re\|s\gamma_0^{\re}v\|_{L^{\infty}(\Gamma^{\re})}\\
	&= \re^{-1}\|v(s,\vec{0})-\gamma_0^{\re}v\|_{L^2(\Gamma^{\re})}+\re\|s\gamma_0^{\re}v\|_{L^{\infty}(\Gamma^{\re})}\\
	&\lesssim \re\|\nabla v\|_{L^{\infty}(B_{\re})}+\re \|sv\|_{L^{\infty}(B_{\re})},
\end{align*}
i.e. \eqref{eq:first_estimate}, where we use that $v$ is $C^{\infty}(\overline{B}_1)$. Replacing $v=\mathcal{R}_0(s)f$, it remains to obtain a bound in the rhs explicit in $\|f\|_{L^2(\mathcal{O}_0)}$. We conclude by applying Lemma \ref{lem:rz}, the inequality $\Re s>-c_{\Omega}\re^{-1}$ and using $c_{\mathcal{O}_0}=R_0^{3/2}\max(1, r_0^{-2})$.  
\end{proof}
\begin{lemma}
	\label{lem:dif_operators}
There exists $\re_0$, s.t. for all $0<\re<\re_0$, $s\in \mathbb{C}$, $R>r>1$, it holds 
	\begin{align*}
			\|\mathcal{S}^{\re}-\mathcal{S}^{\re}_{app}\|_{L^2(\Gamma^{\re})\rightarrow L^{2}(B_{r,R})}&\leq C\re^2 R^{3/2}\mathrm{e}^{|\Re s|_{-}(R+\re)}r^{-1}(|s|+r^{-1}),\\
			\|\mathcal{S}^{\re}\|_{L^2(\Gamma^{\re})\rightarrow L^2(B_{r,R})}&\leq C\re R^{3/2}r^{-1}\mathrm{e}^{|\Re s|_{-}(R+\re)}.
	\end{align*}
\end{lemma}
\begin{proof}	
We start by proving the following bound:
	\begin{align*}
			\|\mathcal{S}^{\re}-\mathcal{S}^{\re}_{app}\|_{L^2(\Gamma^{\re})\rightarrow L^{\infty}(B_{r,R})}&\leq C\re \mathrm{e}^{|\Re s|_{-}(R+\re)}r^{-1}(|s|+r^{-1}).
	\end{align*}
Using explicit definitions of the operators and the Cauchy-Schwarz inequality yields
	\begin{align*}				\|&\mathcal{S}^{\re}-\mathcal{S}^{\re}_{app}\|_{L^2(\Gamma^{\re})\rightarrow L^{\infty}(B_{r,R})}^2\leq  		
		 \sup_{\bx\in B_{r,R}}\int_{\Gamma^{\re}}\left|\frac{\mathrm{e}^{-s\|\bx-\by\|}}{4\pi\|\bx-\by\|}-\frac{\mathrm{e}^{-s\|\bx\|}}{4\pi\|\bx\|}\right|^2d\Gamma_{\re}\\
		&\leq \sup_{\bx\in B_{r,R}}\sup_{\by\in B_{\re}}\|\nabla_{\by} \frac{\mathrm{e}^{-s\|\bx-\by\|}}{4\pi\|\bx-\by\|}\|^2\int_{\Gamma^{\re}}\|\by\|^2d\Gamma_{\by}\lesssim \re^4 \frac{\mathrm{e}^{2|\Re s|_{-}(R+\re)}}{r^2}\left(|s|+r^{-1}\right)^2.
	\end{align*}	
The conclusion of the lemma follows from  $\|.\|_{L^2(B_{r,R})}\lesssim R^{3/2} \|.\|_{L^{\infty}(B_{r,R})}$.

The proof of the bound on $\mathcal{S}^{\re}$ follows the same ideas, thus we omit it. 
\end{proof}
Now we have all the ingredients to prove Proposition \ref{prop:lf}.  
\begin{proof}[Proof of Proposition \ref{prop:lf}]
Applying the triangle inequality to  \eqref{eq:uscre} yields
\begin{align*}
	\|\mathcal{E}^{\re}&\|_{L^2(\mathcal{O}_0)\rightarrow L^2(B_{r,R})}\leq \|\mathcal{S}^{\re}_{app}-\mathcal{S}^{\re}\|_{L^2(\Gamma^{\re})\rightarrow L^2(B_{r,R})}	\|\sigma^{\re}P_{\vec{0}}\mathcal{R}_0(s)\|_{L^2(\mathcal{O}_0)\rightarrow L^2(\Gamma^{\re})}\\
	&+\|\mathcal{S}^{\re}\|_{L^2(\Gamma^{\re})\rightarrow L^2(B_{r,R})}	\|(\mathbb{S}^{\re}(s))^{-1}\gamma_0^{\re}\mathcal{R}_0(s)-\sigma^{\re}P_{\vec{0}}\mathcal{R}_0(s)\|_{L^2(\mathcal{O}_0)\rightarrow L^2(\Gamma^{\re})}.
\end{align*}
The statement follows by combining the bounds of Lemmas \ref{lem:err_lmbd}, \ref{lem:dif_operators} and \ref{lem:estsigma},  using $|\Re s|_{-}\re<c_{\Omega}$ and definitions of $C_{r,R}$ and $C_{\mathcal{O}_0}$ from Proposition \ref{prop:bounds}.
\end{proof}
\subsection{Proof of the time domain bounds}
\label{sec:fin_proof}
We recall an explicit expression of the error between the scattered and the incident fields \eqref{eq:ere}:
\begin{align}
	\label{eq:ere2}
	e^{\re}(t)&=\frac{1}{2\pi i}\int_{\Re s=\mu}\frac{\mathrm{e}^{st}}{(\mu_*-s)^p}(\mathcal{D}^{\re}_{app}(s)-\mathcal{D}^{\re}_D(s))(su_0^p+v_0^p)ds,
\end{align}
where $p\geq 2, \, 0<\mu<\mu_*.$
We have in particular, with $f_p(s):=su_0^p+v_0^p$, 
\begin{align*}		\mathbb{1}_{K_{\operatorname{ff}}}e^{\re}(t)&=\frac{1}{2\pi i}\int_{\Re s=\mu}\frac{\mathrm{e}^{st}}{(\mu_*-s)^p}\mathbb{1}_{K_{\operatorname{ff}}}(\mathcal{D}^{\re}_{app}(s)-\mathcal{D}^{\re}_{D}(s))f_p(s) ds.
\end{align*}
We deform the contour to $\mu=0$ (justified by Proposition \ref{prop:holomorph} and Proposition \ref{prop:bounds}), fix $\theta>0$ sufficiently small so that Proposition \ref{prop:lf} can be applied, and split
\begin{align*}		\mathbb{1}_{K_{\operatorname{ff}}}e^{\re}(t)&=\frac{1}{2\pi }\int_{-\theta/\re}^{\theta/ \re}\frac{\mathrm{e}^{-i\omega t}}{(\mu_*+i\omega)^p}\mathbb{1}_{K_{\operatorname{ff}}}(\mathcal{D}^{\re}_{app}(-i\omega)-\mathcal{D}^{\re}_{D}(-i\omega))f_p(-i\omega) d\omega\\
	&+\frac{1}{2\pi }\int_{|\omega|>\theta/\re}\frac{\mathrm{e}^{-i\omega t}}{(\mu_*+i\omega)^p}\mathbb{1}_{K_{\operatorname{ff}}}(\mathcal{D}^{\re}_{app}(-i\omega)-\mathcal{D}^{\re}_{D}(-i\omega))f_p(-i\omega) d\omega.
\end{align*}
Applying the triangle inequality to both integrals, the bound of Proposition \ref{prop:lf} to the first integral and of  Proposition \ref{prop:hf} with $\ell=1$ to the second integral yields
\begin{align*}
	\|&e^{\re}(t)\|_{L^2(K_{\operatorname{ff}})}\lesssim \re^2C_{r_{\operatorname{ff}}, R_{\operatorname{ff}}}c_{\mathcal{O}_0}\left( \int_{0}^{\frac{\theta}{\re}}\frac{1}{\mu_*^p+\omega^p}(1+\omega+r_{\operatorname{ff}}^{-1})(\omega\|u_0^p\|+\|v_0^p\|)d\omega\right.\\
	&+\left.\int_{\frac{\theta}{\re}}^{+\infty}\frac{\omega}{\mu_*^p+\omega^p}(1+\omega\re)^3(\omega\|u_0^p\|+\|v_0^p\|)d\omega\right)\lesssim \re^2C_{r_{\operatorname{ff}}, R_{\operatorname{ff}}}\max(1, r_{\operatorname{ff}}^{-1})c_{\mathcal{O}_0}E^{p}_0,
\end{align*}
whenever $p>6$. The conclusion follows with Remark \ref{rem:bmap} and $p=k_{reg}\geq 7$.

\bibliographystyle{plain}

\bibliography{references}

\end{document}